\documentclass{pspum-l}

   \usepackage{url}
  \usepackage{amssymb}

  \input xypic
  \xyoption{all}

 \usepackage{times}
 
  \usepackage{amsmath, amsthm, amsfonts}

\newtheorem{theorem}{Theorem}[section]

 \newtheorem{proposition}[theorem]{Proposition}
  
  \newtheorem{corollary}[theorem]{Corollary}

\theoremstyle{definition}
\newtheorem{definition}[theorem]{Definition}

\theoremstyle{remark}
\newtheorem{remark}[theorem]{Remark}

\numberwithin{equation}{section}

  \renewcommand{\cH}{{\mathcal H}}
  
  \newcommand{\cK}{{\mathcal K}}
  \newcommand{\cQ}{{\mathcal Q}}
  \newcommand{\cA}{{\mathcal A}}
  \newcommand{\cE}{{\mathcal E}}

  \renewcommand{\cD}{{\mathcal D}}
  \newcommand{\cP}{{\mathcal P}}
  
  \newcommand{\cG}{{\mathcal G }}

\newcommand{\cFr}{{\mathcal Fr }}

 \newcommand{\Cliff}{{\bf  Cliff}}

\newcommand{\cKa}{K_{{\bf an} }}
\newcommand{\cKt}{K_{\bf{top}} }

\newcommand{\hKa}{K^{\bf{an}} }
\newcommand{\hKt}{K^{\bf{top}} }
\newcommand{\hKg}{K^{\bf{geo}} }

   \newcommand{\ba}{\begin{eqnarray}}
   \newcommand{\na}{\end{eqnarray}}
   \newcommand{\ban}{\begin{eqnarray*}}
   \newcommand{\nan}{\end{eqnarray*}}

     \newcommand{\vlim}{\varinjlim}
   \newcommand{\s}{{\mathfrak s}}

  \newcommand{\CC}{{\mathbb C}}
  \newcommand{\RR}{{\mathbb R}}
  \newcommand{\ZZ}{{\mathbb Z}}
  \newcommand{\QQ}{{\mathbb Q}}
  \newcommand{\NN}{{\mathbb  N}}
    \newcommand{\KK}{{\mathbb  K}}

 \newcommand{\CL}{{\bf  Cliff}}
  \newcommand{\ind}{{\bf  Index }}

\newcommand{\BSpin}{{\bf  BSpin}}
 
  \newcommand{\BSO}{{\bf  BSO}}

   \newcommand{\Fred}{{\bf  Fred}}
 
  \renewcommand{\a}{\alpha}
  \renewcommand{\b}{\beta}
   \renewcommand{\i}{\iota}
  
  \newcommand{\eps}{\epsilon}
  
  \newcommand{\pa}{\partial}
    \newcommand{\disp}{\displaystyle}

  \newcommand{\nin}{\noindent}

  \newcommand{\<}{\langle}
  \renewcommand{\>}{\rangle}

  \def\cancel#1#2{\ooalign{$\hfil#1\mkern1mu/\hfil$\crcr$#1#2$}}
\def\Dirac{\mathpalette\cancel D}

\begin{document}

\title{Riemann-Roch and  index formulae  in twisted $K$-theory}

  \author{Alan L. Carey}
\address{Mathematical Sciences Institute\\
  Australian National University\\
  Canberra ACT 0200 \\
  Australia}
  \email{acarey@maths.anu.edu.au}

  \author{Bai-Ling Wang}
  \address{Department of Mathematics\\
  Australian National University\\
  Canberra ACT 0200 \\
  Australia}
  \email{wangb@maths.anu.edu.au}

%    General info
\subjclass{Primary 54C40, 14E20; Secondary 46E25, 20C20}
\date{}

%\dedicatory{}

\keywords{Twisted  $K$-theory,  twisted $K$-homology, twisted Riemann-Roch,  twisted index theorem, D-brane charges}

   \begin{abstract} In this paper, we establish the Riemann-Roch theorem in twisted  $K$-theory extending our earlier results. We also give a careful summary of
   twisted geometric cycles explaining in detail some subtle points in the theory. As an application, we prove a twisted index formula  and  show that  D-brane charges in Type I and Type II  string theory are classified by twisted KO-theory and twisted  $K$-theory respectively in the presence of $B$-fields as proposed by Witten. 
  
  \end{abstract}

\maketitle

\tableofcontents

\newpage

\section{Introduction}

\subsection{String geometry}
We begin by giving a short discussion of the physical background.
Readers uninterested in this motivation may move to the next subsection.
In string theory D-branes were proposed as a mechanism for 
providing boundary conditions for the dynamics of open strings
moving in space-time. Initially they were thought of as
submanifolds. As D-branes themselves can evolve over time one needs to study equivalence relations on the set of D-branes. An invariant of the equivalence class is the topological charge of the
D-brane which should be thought of as an analogue of the Dirac monopole charge as these D-brane charges are associated with 
gauge fields (connections) on vector bundles over the D-brane.
These vector bundles are known as Chan-Paton bundles.

In \cite{MM}  Minasian and Moore made the proposal that D-brane charges should take values in $K$-groups and not in the cohomology of the space-time or the D-brane. However, they proposed a 
cohomological formula for these charges which might be thought of as
a kind of index theorem in the sense that, in general,  index theory associates to a $K$-theory class a number which is given by an integral of a closed differential form. In string theory
there is an additional field on space-time known as the $H$-flux which may be thought of as a global closed three form. Locally it is given by a family of `two-form potentials' known as the $B$-field. 
Mathematically we think of these $B$-fields as defining a degree three integral \v{C}ech class on the space-time, called a `twist'.
Witten \cite{Wit1}, extending \cite{MM}, gave a physical argument for the idea that
D-brane charges should be elements of $K$-groups and, in addition, proposed that the D-brane charges in the presence of a twist should take values in twisted $K$-theory (at least in the case where
the twist is torsion). The mathematical ideas he relied on were due to Donovan and Karoubi
\cite{DK}. Subsequently Bouwknegt and Mathai \cite{BouMat} 
extended Witten's proposal to the non-torsion case using ideas from \cite{Ros}. A geometric model (that is, a `string geometry' picture) for some of these string theory constructions and for twisted $K$-theory  was proposed in 
\cite{BCMMS} using the notion of bundle gerbes and bundle gerbe modules.
 Various refinements of twisted  $K$-theory that are suggested by these applications are also described in the article of Atiyah and Segal  \cite{AS1} and we will need to use their results here.

\subsection{Mathematical results}
From a mathematical perspective some immediate questions arise from the physical input summarised above. When there is no twist it is well known that $K$-theory provides the main topological tool for the index theory of elliptic operators. One version of the Atiyah-Singer index theorem due to Baum-Higson-Schick 
\cite{BHS} establishes a relationship between the analytic viewpoint provided by elliptic differential operators and the geometric viewpoint provided by 
the notion of geometric cycle introduced in the fundamental paper of Baum and Douglas \cite{BD2}. The viewpoint that geometric cycles in the sense of \cite{BD2} are a model for D-branes in the untwisted case is expounded in
\cite{RS,RSV,Sz}. Note that in this viewpoint D-branes are no longer submanifolds but the images of manifolds under a smooth map.

It is thus tempting to conjecture that there is an analogous picture of D-branes as a type of geometric cycle in the twisted case as well.
More precisely we ask the question of whether there is a way to formulate the notion of `twisted geometric cycle'  (cf \cite{BD1} and \cite{BD2}) and to prove an index theorem in the spirit of \cite{BHS} for twisted $K$-theory. This precise question was answered in the positive in \cite{Wang}. It is important to emphasise that string geometry ideas from \cite{FreWit}
played a key role in finding the correct way to generalise \cite{BD1}.

Our purpose here in the present paper is threefold.
First, we explain the results in \cite{Wang} (see Section 5)
in a fashion that is more aligned to the string geometry viewpoint.
Second, we prove an analogue of the Atiyah-Hirzebruch  Riemann-Roch formula in twisted  $K$-theory by extending the results and approach of \cite{CMW}. An interesting by-product of our approach
in Section 5 is a discussion of the Thom class in
twisted K theory. Third, in Section 6 we prove  an index formula 
using our twisted Riemann-Roch theorem. It will be clear from our approach to this twisted index theory that our twisted geometric cycles provide a geometric model for D-branes and we give details in Section 7.

Our main new results are stated as
two theorems,  Theorem \ref{RR} (twisted Riemann-Roch)
and Theorem \ref{index} (the index pairing).
We remark that the Minasian-Moore formula \cite{MM}
arises from the fact that the index pairing they discuss may be regarded as a quadratic form on $K$-theory. In the twisted index formula that we establish, the pairing is asymmetric and may be
thought of as a bilinear form, from which there is no obvious way to
extract a twisted analogue of the Minasian-Moore formula. Nevertheless we interpret our results in terms of
the physics language in Section 7 explaining the link to Witten's original ideas 
on D-brane charges.

%\section{D-brane charges}

%In string theory $D$-branes were proposed as a mechanism for 
%providing boundary conditions for the dynamics of open strings.

\section{Twisted  $K$-theory: preliminary review}

\subsection{Twisted  $K$-theory: topological and analytic definitions}

We begin by reviewing the notion of a `twisting'.
Let $\cH$ be an  infinite dimensional, complex and separable Hilbert space. We shall consider locally trivial  principal $PU(\cH)$-bundles over a
paracompact  Hausdorff topological space $X$,    the structure group $PU(\cH)$ is equipped with the norm topology.  
The projective unitary group $PU(\cH)$ with the topology induced by the norm topology on $U(\cH)$ (Cf. \cite{Kui})  has the homotopy type of an 
Eilenberg-MacLane space $K(\ZZ, 2)$. 
The  classifying space of $PU(\cH)$, denoted $BPU(\cH)$,     is a $K(\ZZ, 3)$. 
The set of isomorphism classes of 
principal $PU(\cH)$-bundles over   $X$ is given by  (Proposition 2.1 in \cite{AS1}) homotopy classes of maps 
from $X$ to any $K(\ZZ,3)$ and there is a canonical identification
$$[X, BPU(\cH)] \cong H^3(X, \ZZ).$$

  A twisting of  complex $K$-theory on $X$ is given by   a  continuous map $\a: X\to K(\ZZ, 3)$. For such a twisting,
   we can associate a  canonical principal $PU(\cH)$-bundle
 $\cP_\a$ through  the usual pull-back construction from the universal 
 $PU(\cH)$ bundle denoted by $EK(\ZZ, 2)$, as summarised by the diagram:
\ba\label{bundle}
\xymatrix{\cP_\a  \ar[d]   \ar[r]  &
EK(\ZZ, 2) 
  \ar[d]  \\
X \ar[r]_\a  & K(\ZZ, 3). } 
\na
   We will use $ PU(\cH)$ as a group model for a $K(\ZZ,2)$.
   We write ${\bf Fred}(\cH)$ for the connected component of the identity of the space of Fredholm operators on $\cH$
   equipped with the norm topology. There is a base-point preserving action 
   of $PU(\cH)$ given by the conjugation action of $U(\cH)$ on $\Fred(\cH)$:
\ba\label{action}
PU(\cH) \times  \Fred (\cH) \longrightarrow \Fred (\cH).
\na 

The action (\ref{action})  defines an associated bundle over $X$
which we  denote by  
\[
\cP_\a (\Fred) = \cP_\a\times_{PU(\cH)} {\mathbf{Fred}}(\cH)
\]
%the bundle of based spectra over $X$ with fiber the   $K$-theory spectrum,  
We write  $\{ \Omega^n_X \cP_\a(\Fred) = 
 \cP_\a\times_{PU(\cH)} \Omega^n \Fred \}$  for the   fiber-wise iterated loop spaces.

 \begin{definition} The (topological)  twisted K-groups  of $(X, \a)$ are defined to be
\[
K^{-n}(X, \a) := \pi_0\bigl( C_c(X, \Omega^n_X \cP_\a(\Fred))\bigr),
\]
the  set of homotopy classes of compactly supported sections (meaning they are the identity operator in $\Fred$  off a compact set) of the bundle of $\cP_\a (\Fred)$.
 \end{definition}

Due to Bott periodicity, we only have two different  twisted K-groups $K^0(X, \a)$ and
$K^1(X, \a)$. 
Given  a closed subspace $A$  of $X$, then $(X, A)$ is   a pair of topological spaces, and
  we define relative 
 twisted K-groups  to be 
 \[
 K^{ev/odd}(X, A;  \a) := K^{ev/odd}(X-A, \a).
 \]
 
Take a pair of twistings $\a_0, \a_1: X \to K(\ZZ, 3)$, 
and a map $\eta: X\times [1, 0] \to K(\ZZ, 3)$ which is a homotopy 
between  $\a_0$ and $\a_1$, represented diagrammatically by
\[
  \xymatrix{X  \ar@/^2pc/[rr]_{\quad}^{\a_0}="1"
\ar@/_2pc/[rr]_{\a_1}="2"
&& K(\ZZ, 3).
\ar@{}"1";"2"|(.2){\,}="7"
\ar@{}"1";"2"|(.8){\,}="8"
\ar@{==>}"7" ;"8"^{\eta} }
  \]
 Then there is a canonical isomorphism 
$\cP_{\a_0} \cong \cP_{\a_1}$  induced by $\eta$.  This canonical isomorphism
determines a canonical  isomorphism on twisted K-groups 
\ba\label{iso:eta}\xymatrix{ 
\eta_*:  \  K^{ev/odd}(X,  \a_0) \ar[r]^{\cong}  & K^{ev/odd}(X,  \a_1),}
\na
This isomorphism $\eta_*$  depends only on the homotopy class of $\eta$.     
The set of  homotopy classes of maps between  $\a_0$ and $\a_1$ is labelled by 
$[X, K(\ZZ,   2)]$. Recall the first Chern class  isomorphism
\[
{\bf Vect}_1(X) \cong [X, K(\ZZ,   2)] \cong H^2(X, \ZZ)
\]
where ${\bf Vect}_1(X)$ is 
 the set of equivalence classes of complex line bundles on $X$.  We remark that the isomorphisms 
induced by  two different homotopies 
between  $\a_0$ and $\a_1$   are related through an action of complex line bundles.

  Let $\cK $ be the $C^*$-algebra  of  compact 
  operators on $\cH$.  
The isomorphism $PU(\cH) \cong Aut ( \cK)$ via the conjugation action of the unitary group
$U(\cH)$ provides an action of a
$K(\ZZ, 2)$ on the $C^*$-algebra  $\cK$. Hence, any
$K(\ZZ, 2)$-principal bundle $\cP_\a$ defines a locally trivial  bundle of compact operators,
denoted by 
$
\cP_\a(\cK) = \cP_\a\times_{PU(\cH)} \cK.
$

Let  $C_0(X, \cP_\a(\cK))$ be  the $C^*$-algebra of sections of 
$\cP_\a(\cK)$ vanishing at infinity.   Then $C_0(X, \cP_\a(\cK)$ is  the (unique up to isomorphism) stable separable complex  continuous-trace $C^*$-algebra over $X$ with 
Dixmier-Douday class $[\a] \in H^3(X, \ZZ)$ (here we identify the \v{C}ech cohomology of $X$ with its
singular cohomology, cf \cite{Ros} and \cite{AS1}).

\begin{theorem} \label{twited:K:top=ana} ( \cite{AS1} and \cite{Ros}) The topological  twisted K-groups $K^{ev/odd}(X, \a)$
are canonically isomorphic to analytic  $K$-theory  of  the
 $C^*$-algebra  $C_0(X, \cP_\a(\cK))$
\[
K^{ev/odd}(X, \a) \cong K_{ev/odd} (C_0(X, \cP_\a(\cK))) 
\]
where the latter group is  the algebraic  $K$-theory of  $C_0(X, \cP_\a(\cK))$, defined to be
\[
\vlim_{k\to \infty} \pi_1\bigl(GL_k(C_0(X, \cP_\a(\cK)))\bigr).
\]
Note that   the algebraic  $K$-theory of  $C_0(X, \cP_\a(\cK))$ is isomorphic to  KasparovÕs $KK$-theory  
(\cite{Kas} and \cite{Kas1})
\[
KK^{ev/odd}(\CC, C_0(X, \cP_\a(\cK)).
\]

\end{theorem}

It is important to recognise that these groups are only defined up to isomorphism by the Dixmier-Douady class $[\a] \in H^3(X, \ZZ)$.
To distinguish these two equivalent definitions of twisted  $K$-theory if needed, we will write
\[
\cKt^{ev/odd}(X, \a) \qquad {\text{and} }  \qquad \cKa^{ev/odd}(X, \a)
\]
 for the topological and analytic twisted  $K$-theories of
$(X, \a)$ respectively.
Twisted  $K$-theory  is a 2-periodic  {\em generalized cohomology theory}: 
 a contravariant functor on the category
 consisting of  pairs $(X, \alpha)$, with the twisting $\a: X\to K(\ZZ, 3)$, to the category of
 $\ZZ_2$-graded abelian groups.    Note that a  morphism  between two pairs
 $(X, \a)$ and $(Y, \b)$ is a continuous map $f: X\to Y$ such that $ \b  \circ f =\a$. 
 
\subsection{Twisted  $K$-theory for torsion twistings}
 There are some subtle issues in twisted  $K$-theory and to handle these we
 have chosen to use  the language of bundle gerbes, connections and curvings
 as explained in \cite{Mur}.  We explain first the  so-called `lifting bundle gerbe' $\cG_\a$ \cite{Mur} associated to the principal $PU(\cH)$-bundle $\pi: \cP_\a\to X$ and 
 the central extension
\ba\label{cen:ext}
1\to U(1) \longrightarrow U(\cH) \longrightarrow PU(\cH) \to 1.
\na
This is constructed by starting with $\pi:  \cP_\a\to X$, forming the fibre product $\cP_\a  ^{[2]}$
which is a  groupoid
\[\xymatrix{
\cP_\a^{[2]} = \cP_\a \times_X \cP_\a  \ar@< 2pt>[r]^{\qquad  \pi_1} \ar@< -2pt>[r]_{ \qquad  \pi_2}  &  \cP_\a}
\]
 with 
source and range maps $\pi_1: (y_1, y_2) \mapsto y_1$ and $\pi_2: (y_1, y_2)\mapsto y_2$. There is an obvious map from each fiber of $\cP_\a^{[2]}$ to $PU(\cH)$ and so we can define
the fiber of  $\cG_\a$ over a point in $\cP_\a^{[2]}$ by pulling back the 
fibration (\ref{cen:ext}) using this map. 
This endows $\cG_\a$ with  a groupoid 
structure (from the multiplication in $U(\cH)$) and in fact it is a $U(1)$-groupoid extension of $\cP_\a  ^{[2]}$.

A torsion twisting $\a$ is a map $\a:  X\to K(\ZZ, 3)$ representing a torsion class in $H^3(X, \ZZ)$. Every torsion twisting arises from  a principal $PU(n)$-bundle 
$ \cP_\a(n)$ with its classifying map
\[
X\to BPU(n), 
\]
or a principal $PU(\cH)$-bundle with a reduction to 
$PU(n) \subset PU(\cH)$.  
For a torsion twisting $\a: X\to BPU(n) \to BPU(\cH)$,  the corresponding lifting bundle gerbe $\cG_a$ 
\ba\label{bg:sigma}
\xymatrix{
 \cG_\a \ar[d]& \\
 \cP_\a(n)^{[2]}\ar@< 2pt>[r]^{\pi_1} \ar@< -2pt>[r]_{\pi_2}
  & \cP_\a(n)    \ar[d]^{\pi}\\
   &M}
\na
is defined by $ \cP_\a(n)^{[2]}\cong  \cP_\a(n) \rtimes PU(n) \rightrightarrows \cP_\a(n)$  (as a groupoid) and the   central extension
\[
1\to U(1) \longrightarrow U(n) \longrightarrow PU(n) \to 1.
\]

There is an Azumaya bundle associated to $\cP_\a(n)$
arising naturally from the $PU(n)$ action on the $n\times n$
matrices. We denote this associated Azumaya 
 bundle by $\cA_\a$.    An $\cA_\a$-module is a complex vector
  bundle $\cE$  over $M$ with a fiberwise $\cA_\a$ action
   $$
   \cA_\sigma \times_M \cE \longrightarrow \cE.
   $$
The  $C^*$-algebra of  continuous  sections  of $\cA_\a$, vanishing at infinity  if $X$ is non-compact,  is Morita equivalent
 to a continuous trace $C^*$-algebra $C_0(X, \cP_\a(\cK))$. Hence there is an isomorphism 
between  $K^0 (X, \a)$ and the  $K$-theory of the bundle modules of $\cA_a$.

There is  an equivalent definition of  twisted  $K$-theory using bundle gerbe modules (Cf. \cite{BCMMS} and \cite{CW1}).  A bundle gerbe module $E$ of $\cG_\a$
 is   a complex vector bundle $E$ over $\cP_\a(n)$ with
a groupoid action of $\cG_\a$, i.e., an isomorphism
\[
\phi: \cG_\a \times_{(\pi_2,p )} E \longrightarrow E
\]
where $\cG_\a \times_{(\pi_2,\pi)} E$ is the fiber product of
the source $\pi_2: \cG_\a \to \cP_\a(n)$  and $p: E\to
\cP_\a(n)$ such that
\begin{enumerate}
\item $p\circ \phi (g, v) = \pi_1(g)$ for $(g, v) \in  \cG_\a
\times_{(\pi_2, p)} E$, and $\pi_1$ is the target map of $\cG_\a$.
\item $\phi$ is compatible with the bundle gerbe multiplication
$m: \cG_a \times_{(\pi_2,\pi_1)}\cG_\a \to \cG_\a$,
which means
\[
\phi \circ (id \times \phi) = \phi\circ (m\times id).
\]
\end{enumerate}

 Note that   the  natural representation of $U(n)$ on $\CC^n$
 induces a $\cG_\a$ bundle gerbe module 
 \[
 S_n = \cP_\a(n) \times \CC^n.
 \]
 Here we use the fact that $\cG_\a =  \cP_\a(n) \rtimes U(n) \rightrightarrows \cP_\a(n)$  (as a groupoid). Similarly, 
 the dual representation of  $U(n)$ on $\CC^n$
 induces a $\cG_{-\a}$ bundle gerbe module $S_n^* =  \cP_\a(n) \times \CC^n$.
 Note that 
$
 S^*_n \otimes S_n  \cong  \pi^*\cA_\a
 $
 descends to the  Azumaya bundle $\cA_\a$.
 Given a  $\cG_\a$ bundle gerbe module $E$ of rank $k$, then as a $PU(n)$-equivariant vector bundle, 
 $S^*_n\otimes E$ descends to an $\cA_\a$-bundle over $M$. Conversely, given 
 an $\cA_\a$-bundle $\cE$ over $M$, $S_n\otimes_{\pi^*\cA_\a} \pi^* \cE$ defines a 
 $\cG_\a$ bundle gerbe module. These two constructions are inverse to each other due to the fact that
 \[
 S_n^* \otimes (S_n \otimes_{\pi^*\cA_\a} \pi^* \cE) \cong (S_n^* \otimes S_n ) \otimes_{\pi^*\cA_\a} \pi^* \cE \cong \pi^*\cA_\a \otimes_{\pi^*\cA_\a}  \pi^* \cE  \cong \pi^* \cE .
 \]
   Therefore,
 there is a natural equivalence between the category of 
 $\cG_\a$ bundle gerbe modules and the category of  $\cA_\a$ bundle modules, as discussed in \cite{CW1}.  In summary, we have the following proposition.
 
  \begin{proposition} (\cite{BCMMS}\cite{CW1}) For a torsion twisting $\a: X\to BPU(n) \to BPU(\cH)$,   twisted $K$-theory $K^0(X, \a) $ has  another   two equivalent descriptions:
  \begin{enumerate}
 \item   the Grothendieck group of the category of  $\cG_\a$ bundle gerbe modules.
\item   the Grothendieck group of the category of $\cA_\sigma$  bundle modules.
\end{enumerate}
  \end{proposition}

One important  example of torsion twistings comes  from real oriented vector bundles.  Consider an  oriented real vector bundle $E$ of even rank over $X$
with a fixed fiberwise inner product.  
Denote by 
 \[
 \nu_E: X\to \BSO(2k)
 \]
 the classifying map of $E$. The following twisting
 \[
o(E) :=W_3\circ \nu_E: X \longrightarrow  \BSO(2k)  \longrightarrow   K(\ZZ, 3), 
\]
will be called the orientation twisting  associated to $E$.  Here $W_3$ is the classifying map
of the principal $\mathbf{BU}(1)$-bundle $\BSpin^c (2k) \to \BSO (2k)$. Note that 
 the orientation twisting $o(E)$ is null-homotopic if and only if $E$ is K-oriented.

  \begin{proposition} \label{Cliff} Given an  oriented real vector bundle $E$ of even rank over $X$ with 
  an orientation twisting $o(E)$, then there is a canonical isomorphism 
  \[
  K^0(X, o(E)) \cong K^0(X, W_3(E))
  \]
  where $K^0(X, W_3(E))$ is the  $K$-theory of the Clifford modules associated to the bundle  $\Cliff(E)$ of 
  Clifford  algebras.
  \end{proposition}
       \begin{proof}
 Denote by $\cFr$
the frame bundle of $V$, the principal $SO(2k)$-bundle of
positively oriented orthonormal frames, i.e.,
$$
E= \cFr\times_{\rho_{2n}} \RR^{2k}, $$
where $\rho_n$ is the standard representation of $SO(2k)$ on $\RR^n$.
 The lifting bundle gerbe  associated to the frame bundle and
the central   extension
  \[
  1\to U(1) \longrightarrow Spin^c(2k) \longrightarrow SO(2k) \to 1
  \]
  is called the $Spin^c$ bundle gerbe  $\cG_{W_3(E)}$ of $E$, whose  Dixmier-Douady invariant is  given by  the integral third Stiefel-Whitney class   $W_3(E)\in H^3(X, \ZZ)$.  The canonical representation of $Spin^c(2k)$ gives a natural inclusion
  \[
  Spin^c(2k) \subset U(2^{k})
  \]
  which induces a commutative diagram
  \[
  \xymatrix{
U(1)\ar[r] \ar[d]^{=} &   Spin^c(2k) \ar[r] \ar[d] & SO(2k) \ar[d]\\
 U(1)\ar[r] \ar[d]^{=}&  U(2^k) \ar[r]  \ar[d]& PU(2^k) \ar[d]\\
 U(1)\ar[r] &  U(\cH) \ar[r] & PU(\cH).
  }
  \]
  This provides a  reduction of the principal $PU(\cH)$-bundle  $\cP_{o(E)}$. 
  The associated  bundle of Azumaya algebras 
  is in fact the  bundle of Clifford algebras, whose bundle modules are called Clifford modules (\cite{BGV}). 
 Hence, there
  exists a canonical isomorphism between $K^0(X, o(E))$ and the  $K$-theory of   the Clifford modules associated to the bundle  $\Cliff(E)$.
  \end{proof}

 \subsection{Twisted  $K$-theory: general properties}
 
Twisted  $K$-theory satisfies  the following  properties whose proofs are rather standard 
 for  a 2-periodic generalized cohomology theory (\cite{AS1} \cite{CW1} \cite{Kar} \cite{Wang}).  (Note that when we write $(X,A)$ for a pair of spaces we assume $A\subset X$.)
 \begin{enumerate}

  \item[(I)] ({\bf The homotopy axiom}) If two morphisms  $f, g: (Y, B)\to (X, A)$ are homotopic through
a map $\eta: (Y\times [0, 1] , B \times [0, 1]) \to (X, A)$, written  in terms of  the following homotopy
commutative diagram 
\[
\xymatrix{(Y, B) \ar[d]_{g} \ar[r]^{f} &
(X, A)
 \ar@2{-->}[dl]_{\eta} \ar[d]^{\a} \\
(X, A) \ar[r]_\a  & K(\ZZ, 3), } 
\]
then we have the following commutative diagram
\[
\xymatrix{
& K^{ev/odd}(X, A; \a)\ar[dr]^{g^*} \ar[dl]^{f^*} &\\
K^{ev/odd}(Y, B; \a\circ f)\ar[rr]^{\eta_*} && K^{ev/odd}(Y, B; \a\circ g).}
\]
Here $\eta_*$ is the canonical isomorphism induced
by the homotopy $\eta$.
\item[(II)]   ({\bf The exact  axiom}) For any pair $(X, A)$ with a twisting $\a: X\to K(\ZZ, 3)$, there exists the following
 six-term exact sequence
 \[\xymatrix{
 K^0 (X, A; \a) \ar[r]& K^0  (X, \a)\ar[r]& K^0 (A,\a |_A)
  \ar[d]\\
K^1 (A, \a |_A)  \ar[u] & K^1  ( X, \a) \ar[l]&
    K^1  (X, A; \a) \ar[l]
   }
 \]
here $\a |_A$ is the composition of the inclusion and $\a$.

  \item[(III)]  ({\bf The excision axiom})  Let $(X, A)$ be a pair of spaces and let 
$U \subset A$ be a subspace  such that  the closure $\overline{U} $ is contained in the interior of
 $A$. Then the inclusion $\i: (X-U, A-U) \to (X, A)$ induces, for all $\a: X\to K(\ZZ, 3)$,  an 
isomorphism 
\[
K^{ev/odd}  (X, A; \a)  \longrightarrow K^{ev/odd}  (X-U, A-U; \a\circ \i).
\]

  \item[(IV)]  ({\bf Multiplicative property}) Let $\a, \b: X\to K(\ZZ, 3)$ be a pair of twistings on $X$. Denote
 by $ \a+ \b$ the new twisting defined by the following map\footnote{In terms of bundles of projective  Hilbert space,  this operation corresponds to the 
Hilbert space tenrsor product, see \cite{AS1}.}  
\ba\label{product}
\xymatrix{
\a   +\b :  \quad   X\ar[r]^{(\a   , \b )  \ \  } & K(\ZZ, 3) \times K(\ZZ, 3) \ar[r]^{
\qquad m}   & 
K(\ZZ, 3), }
\na
where $m$ is defined as follows
\[
B PU(\cH) \times BPU(\cH) \cong  B( PU(\cH)  \times  PU(\cH) )    \longrightarrow B PU(\cH), 
\]
for a fixed  isomorphism $\cH \otimes \cH \cong \cH$. 
   Then there is a canonical
multiplication
\ba\label{multi:add}
K^{ev/odd} (X, \a)  \times K^{ev/odd} (X,   \b) \longrightarrow K^{ev/odd} (X, \a + \b ),
\na
which defines a $K^0(X)$-module structure  on   twisted K-groups $K^{ev/odd} (X, \a)$.

 \item[(V)]   ({\bf Thom isomorphism})  Let $\pi: E\to X$ be an oriented real vector bundle of rank $k$  over $X$, then there is a canonical isomorphism, for any twisting $\a: X\to K(\ZZ, 3)$,
 \ba\label{Thom:CW}
 K^{ev/odd}(X, \a +o_E )  \cong K^{ev/odd} (E, \a \circ  \pi),
 \na
with the grading shifted by $k (mod \ 2 )$. 

\nin 
 
\item[(VI)]   ({\bf The push-forward map}) 
  For any differentiable map $f: X\to Y$ between two smooth manifolds $X$ and $Y$,  let $\a: Y \to K(\ZZ, 3)$
  be a twisting. Then there is a canonical  push-forward  homomorphism 
  \ba\label{push:forward}
  f^K_!:   \qquad  K^{ev/odd} \bigl(X, ( \a \circ f)  +  o_{f}  \bigr) \longrightarrow K^{ev/odd}(Y, \a),
  \na
  with the grading shifted  by $n \ mod (2)$ for $n= \dim (X)+ \dim (Y)$. Here  $o_f$ is the  orientation twisting corresponding  to  the bundle $TX\oplus f^* TY$ over $X$.  
  
  \item[(VII)]  ({\bf  Mayer-Vietoris    sequence})
  If $X$ is covered by two  open subsets $U_1$ and $U_2$ with a twisting  $\a: X\to K(\ZZ, 3)$,  then
    there is a Mayer-Vietoris  exact sequence
 \[
 \xymatrix{
 K^0 (X, \a)\ar[r]& K^1  (U_1\cap U_2, \a_{12})\ar[r]& K^1 (U_1,\a_1 )\oplus K^1( U_2,\a_2 )
  \ar[d]\\
 K^0  (U_1, \a_1) \oplus K^0(U_2,\a_2 )\ar[u] & K^0  (U_1\cap U_2,\a_{12} ) \ar[l]&
   K^1 (X, \a) \ar[l]
   }
 \]
   where $\a_1$, $\a_2$ and $\a_{12}$ are the restrictions of $\a$ to
   $U_1$, $U_2$ and $U_1\cap U_2$ respectively.
\end{enumerate}

\section{Twisted $K$-homology}

Complex $K$-theory, as a generalized cohomology theory on a CW complex,  is developed by Atiyah-Hirzebruch using complex vector bundles. It is representable in the sense that
there exists a classifying space  $\ZZ\times BU(\infty) $, where $BU(\infty) =\vlim_{k}BU(k)$,
such that 
\[
K^0(X) = [X, \ZZ\times BU(\infty)]
\]
for any finite CW complex $X$.  The classifying space for complex  $K$-theory is 
referred to as the 
 $BU(\infty)$-spectrum with  even term $\ZZ\times BU(\infty)$  and odd term  $U(\infty)$. They are also called the `complex K-spectra' in the literature.  
The  advantage of using spectra is that there is a natural definition of a homology
theory associated to a  classifying space of each generalized cohomology theory. 
Hence,  the topological $K$-homology of a CW complex $X$, dual to complex $K$-theory,  is defined by the following stable homotopy groups
 \[
\hKt_{ev} (X)   =  \disp{\vlim_{k\to\infty}}  \pi_{2k} (BU(\infty)\wedge X^+)
\]
and 
  \[
 \hKt_{odd} (X, \a) =  \disp{\vlim_{k\to\infty}} \pi_{2k+1} (BU(\infty)\wedge X^+).
 \]
Here $X^+$ is the space $X$ with one point added as a based point, and   the wedge product 
of two based CW complexes  $(X, x_0)$ and $(Y, y_0)$
is defined to be
\[
X\wedge Y = \dfrac{X\times Y} {(X\times \{y_0\} \cup \{x_0\} \times Y)}.
\]

 All the properties of $K$-homology, as a generalized homology theory, can be obtained in a natural
 way see for example in \cite{Swi}. There are two other equivalent definitions of $K$-homology,   called analytic $K$-homology developed by Kasparov, and geometric $K$-homology
 by  Baum and Douglas. We now give a brief review of these two definitions.
 
Kasparov's  analytic $K$-homology    $KK^{ev/odd}(C(X), \CC)$ is 
generated by unitary equivalence classes of (graded)  Fredholm modules  over $C(X)$ modulo an operator homotopy
relation (\cite{Kas} and \cite{HigRoe}).   For brevity we will use the notation
$\hKa_{ev/odd}(X)$ for this $K$-homology.  A cycle for  $\hKa_{0}(X)$, also called a $\ZZ_2$-graded Fredholm module, consists of a triple 
$(\phi_0\oplus, \phi_1, \cH_0\oplus \cH_1, F)$, where \begin{itemize}
\item $\phi_i: C(X) \to B(\cH_i) $ is a representation 
of  $C(X)$  on a  separable Hilbert space  $\cH_i$;
\item $F: \cH_0 \to \cH_1$ is a  bounded
operator such that 
\[
\phi_1(a) F - F\phi_0(a), \qquad \phi_0(a) (F^*F -Id) \qquad \phi_1(a) (FF^* -Id)
\]
are compact operators for all $a\in C(X)$. 
\end{itemize}
A cycle for  $\hKa_1(X)$,  also called a trivially graded or odd Fredholm module,   consists of a pair
 $(\phi, F)$, where \begin{itemize}
\item $\phi: C(X) \to B(\cH) $ is a representation 
of  $C(X)$  on a separable Hilbert space   $\cH$;
\item $F$ is a bounded self-adjoint operator on $\cH$ such that
\[
\phi(a) F -F \phi(a),  \qquad \phi(a) (F^2-Id)
\]
are compact operators for all $a\in C(X)$. 
\end{itemize}

In \cite{BD1} and \cite{BD2}, Baum and Douglas gave a geometric definition of $K$-homology using what are now called geometric cycles.   The basic cycles for $ \hKg_{ev} (X) $ (respectively $\hKg_{odd}(X)$)  are triples 
\[
(M, \i, E)
\]
  consisting of even-dimensional (resp. odd-dimensional)  closed smooth manifolds $M$  with a
 given $Spin^c$ structure on the tangent bundle of $M$ together with a continuous map $\i: M\to X$
 and a complex vector bundle $E$ over $M$. The equivalence relation on the  set  of all cycles is generated by 
 the following three steps (see \cite{BD1} for details):
 \begin{enumerate}
\item[(i)] Bordism.
\item[(ii)] Direct sum and disjoint union.
\item[(iii)] Vector bundle modification.
\end{enumerate}
Addition in $\hKg_{ev/odd}(X)$ is given by the disjoint union operation of geometric cycles. 

Baum-Douglas in \cite{BD2} showed that the Atiyah-Singer index theorem is encoded in the following commutative diagram
 \ba\label{BD:origin}
  \xymatrix{ 
& \hKt_{ev/odd} (X )  \ar[dl]_{\cong} \ar[dr]^{\cong}
&\\
\hKg_{ev/odd}(X) \ar[rr]^{\mu} & & \hKa_{ev/odd} (X)}
\na
where $\mu$ is the assembly map assigning an abstract Dirac operator  
\[
\i_* ([\Dirac_M^E]) \in \hKa_{ev/odd} (X)
\]
  to a  geometric cycle $(M, \i, E)$.  
  
  For a  paracompact Hausdorff space $X$ with a twisting $\a: X\to K(\ZZ, 3)$, all these three
  versions of twisted $K$-homology were studied in \cite{Wang}. They are called
there the twisted  topological, analytic and geometric $K$-homologies, and denoted 
respectively by
  $\hKt_{ev/odd} (X, \a)$, $\hKa_{ev/odd} (X, \a)$ and $\hKg_{ev/odd} (X, \a)$.
Our first task in this Section is to review these three definitions, see \cite{Wang} for greater detail.

\subsection{ Topological and analytic definitions of twisted  $K$-homology}  
 
Let $X$ be a CW complex (or paracompact  Hausdorff space) with a twisting $\a: X\to K(\ZZ, 3)$.   Let $\cP_\a$ be the corresponding principal $K(\ZZ, 2)$-bundle.  Any
base-point preserving 
action of a $K(\ZZ, 2)$ on a  space  defines an associated 
bundle by the standard construction. In particular, 
as a classifying space of complex line bundles,
a $K(\ZZ, 2)$ acts on the complex  $K$-theory spectrum   $\KK$ representing
the tensor product by complex line bundles, where 
\[
\KK_{ev} = \ZZ\times BU(\infty), \qquad \KK_{odd} = U(\infty).
\] 
Denote by   $\cP_\a (\KK) = \cP_\a\times_{K(\ZZ, 2)} \KK$
the bundle of based  $K$-theory spectra over $X$.   There is a section of $\cP_\a (\KK) = \cP_\a\times_{K(\ZZ, 2)} \KK$ defined by taking the base points of each fiber. The image of this section can be identified with $X$ and we denote by $ \cP_\a ( \KK)/X$ the quotient space
of $  \cP_\a ( \KK)$ obtained by  collapsing the  image of this section. 
 
 The  stable homotopy groups of  $ \cP_\a ( \KK)/X$ by definition give the  topological  twisted $K$-homology groups $\hKt_{ev/odd} (X, \a)$. 
(There are only two due to Bott periodicity of $\KK$.) Thus we have
 \[
 \hKt_{ev} (X, \a) =  \disp{\vlim_{k\to\infty}} \pi_{2k} \bigl( \cP_\a ( BU(\infty)) /X\bigr)
 \]
 and 
  \[
 \hKt_{odd} (X, \a) =  \disp{\vlim_{k\to\infty}} \pi_{2k+1} \bigl( \cP_\a ( BU(\infty) ) /X\bigr).
 \]
 Here the direct limits are taken by the double suspension
 \[
 \pi_{n+2k} \bigl( \cP_\a ( BU(\infty)) /X\bigr) \longrightarrow  \pi_{n+2k+2} \bigl( \cP_\a (S^2 \wedge   BU(\infty)) /X \bigr)
 \]
 and then followed by the standard map 
 \[
 \xymatrix{
 \pi_{n+2k+2} \bigl(    \cP_\a (S^2 \wedge BU(\infty) )/X \bigr)  \ar[r]^{b\wedge  1} & \pi_{n+2k+2} \bigl(  \cP_\a (BU(\infty) \wedge BU(\infty))/X \bigr) \\
 \qquad  \qquad   \qquad   \ar[r]^m &  \pi_{n+2k+2} \bigl(  \cP_\a ( BU(\infty))/X\bigr) } 
 \]
 where $b: \RR^2\to BU(\infty)$ represents  the Bott generator  in $K^0(\RR^2)\cong \ZZ$, $m$ is the base point preserving
 map inducing the ring structure on  $K$-theory.

 For a relative CW-complex $(X, A)$ with  a twisting $\a: X\to K(\ZZ, 3)$, the relative version of topological twisted $K$-homology, denoted
$\hKt_{ev/odd}(X, A, \a)$, is defined to be  $\hKt_{ev/odd}(X/A, \a)$ where $X/A$ is
the quotient space of $X$ obtained by collapsing $A$ to a point.  Then we have the following
exact sequence
\[
\xymatrix{
\hKt_{odd} (X, A; \a) \ar[r]& \hKt_{ev} (A, \a|_A)\ar[r]& \hKt_{ev} (X,\a)
  \ar[d]\\
\hKt_{odd} (X, \a  )  \ar[u] & \hKt_{odd} ( A, \a|_A) \ar[l]&
    \hKt_{ev}  (X, A; \a) \ar[l],
   }
\]
and the excision properties
\[
\hKt_{ev/odd} (X, B; \a) \cong \hKt_{ev/odd} (A, A-B; \a|_A)
\]
for any CW-triad $(X; A, B)$ with a twisting $\a: X\to K(\ZZ, 3)$.  A triple  $(X; A, B)$ is A CW-triad
if $X$ is a CW-complex, and $A$, $B$ are two subcomplexes of $X$ such that $A\cup B = X$.

For the analytic twisted $K$-homology,  recall that $\cP_\a(\cK)$
is  the associated bundle of compact operators on $X$.  Analytic twisted $K$-homology, denoted by 
$\hKa_{ev/odd}(X,  \a)$, is defined to be
\[
\hKa_{ev/odd} (X,   \a) := KK^{ev/odd}\bigl  (C_0(X, \cP_\a(\cK)), \CC \bigr), 
\]
Kasparov's $\ZZ_2$-graded $K$-homology of
the   $C^*$-algebra $C_0(X, \cP_\a(\cK))$.

For a  relative CW-complex   $(X, A)$ with  a twisting $\a: X\to K(\ZZ, 3)$, the relative version of analytic  twisted $K$-homology  $\hKa_{ev/odd}(X, A, \a)$ is defined to be  $\hKa_{ev/odd}(X-A, \a)$.  Then we have the following
exact sequence
\[
\xymatrix{
\hKa_{odd} (X, A; \a) \ar[r]& \hKa_{ev} (A, \a|_A)\ar[r]& \hKa_{ev} (X,\a)
  \ar[d]\\
\hKa_{odd} (X, \a  )  \ar[u] & \hKa_{odd} ( A, \a|_A) \ar[l]&
    \hKa_{ev}  (X, A; \a) \ar[l],
   }
\]
and the excision properties
\[
\hKa_{ev/odd} (X, B; \a) \cong \hKa_{ev/odd} (A, A-B; \a|_A)
\]
for any CW-triad $(X; A, B)$ with a twisting $\a: X\to K(\ZZ, 3)$.

\begin{theorem}  \label{top=ana} (Theorem 5.1 in \cite{Wang}) There is a natural isomorphism   
\[
\Phi:  \hKt_{ev/odd} (X, \a)   \longrightarrow  \hKa_{ev/odd} (X, \a)
\]
 for any {\bf smooth}  manifold $X$ with a twisting $\a: X \to K(\ZZ, 3)$.
\end{theorem}

The proof of this theorem requires Poincar\'e duality between twisted  $K$-theory and twisted $K$-homology (we describe this duality in the next theorem), and the isomorphism (Theorem \ref{twited:K:top=ana}) between topological twisted  $K$-theory and analytic twisted  $K$-theory.

Fix an isomorphism $\cH \otimes \cH \cong \cH$
 which induces a group homomorphism
 $
 U(\cH) \times U(\cH) \longrightarrow U(\cH)
 $
 whose restriction to the center is the group multiplication on $U(1)$. So we have a 
 group homomorphism
 \[
 PU(\cH) \times PU(\cH) \longrightarrow PU(\cH)
 \]
 which defines  a continuous map, denoted $m_\ast$, of CW-complexes
 \[
 BPU(\cH) \times B PU(\cH)  \longrightarrow  BPU(\cH).
 \]
 As $BPU(\cH)$ is  identified as  $K(\ZZ, 3)$,  we may think of this as  a continuous map taking 
 $
 K(\ZZ, 3) \times K(\ZZ, 3)$ to $K(\ZZ, 3), 
 $
 which can be used to define $\a   + o_X$.

There are natural isomorphisms
from  twisted $K$-homology (topological resp.  analytic) to twisted  $K$-theory (topological resp.  analytic) of a smooth manifold $X$    where
the twisting  is shifted by 
\[
\a \mapsto \a   + o_X
\]
where $ \tau: X \to BSO $ is 
  the classifying map of  the stable tangent space and   $\a   + o_X$ denotes  the map $X \to  K(\ZZ, 3)$, representing the class $[\a]+ W_3(X)$ in $H^3(X, \ZZ)$. 

\begin{theorem}  \label{PD:twisted}  Let $X$ be a smooth manifold 
with a twisting $\a: X\to K(\ZZ, 3)$. 
There exist  isomorphisms  
\[
\hKt_{ev/odd} (X, \a) \cong    \cKt^{ev/odd} (X, \a  +o_X ) 
\]
  and  
\[
\hKa_{ev/odd} (X, \a) \cong    \cKa^{ev/odd} (X, \a  +o_X ) \]
  with the degree shifted by $\dim X (mod\ 2)$. 
\end{theorem}

 Analytic Poincar\'e duality was established in \cite{EEK} and \cite{Tu}, and
 topological Poincar\'e duality was established in \cite{Wang}.
Theorem \ref{top=ana} and the exact sequences for a pair $(X, A)$ imply the following 
corollary.

\begin{corollary}   There is a natural isomorphism   
\[
\Phi:  \hKt_{ev/odd} (X, A,  \a)   \longrightarrow  \hKa_{ev/odd} (X, A,  \a)
\]
 for any {\bf smooth}  manifold $X$ with a twisting $\a: X \to K(\ZZ, 3)$ and a closed submanifold 
 $A\subset X$.
\end{corollary}

\begin{remark} In fact,  Poincar\'e duality  as in Theorem  \ref{PD:twisted} holds for any compact Riemannian
manifold $W$ with boundary $\partial W$ and a twisting $\a:W\to K(\ZZ, 3)$.
This duality takes the following form 
\[
\hKt_{ev/odd} (W, \a) \cong    \cKt^{ev/odd} (W, \partial W, \a  +o_W ) 
\]
  and  
\[
\hKa_{ev/odd} (W, \a) \cong    \cKa^{ev/odd} (X, \partial X, \a  +o_W ) \]
  with the degree shifted by $\dim W (mod\ 2)$. 
From this, we have a natural isomorphism  (\cite{BW})
\[
\Phi:  \hKt_{ev/odd} (X, A,  \a)   \longrightarrow  \hKa_{ev/odd} (X, A,  \a)
\]
 for any  CW pair $(X, A)$ with a twisting $\a: X \to K(\ZZ, 3)$ using the Five Lemma. \end{remark}

\subsection{Geometric cycles  and geometric twisted  $K$-homology}

  Let $X$ be a paracompact Hausdorff space  and let  $\a:  X \longrightarrow  K(\ZZ, 3)$ be
a twisting over $X$.

\begin{definition} Given a smooth oriented  manifold $M$ with a classifying map 
$\nu$ of its  stable normal bundle then we say that  $M$ is an $\a$-twisted $Spin^c$ manifold over $X$   if  
$M$ is equipped with an  $\a$-twisted $Spin^c$ structure, that means,    a continuous map $\i: M\to X$ such that  the following diagram 
\[
\xymatrix{M \ar[d]_{\i} \ar[r]^{\nu} & 
\BSO
 \ar@2{-->}[dl]_{\eta} \ar[d]^{W_3} \\
X \ar[r]_\a  & K(\ZZ, 3), } 
\]
commutes up to a fixed  homotopy $\eta$  from $W_3\circ \nu$ and $\a \circ \i$.  Such an  
 $\a$-twisted $Spin^c$ manifold over $X$ will be denoted by $(M, \nu, \i, \eta)$. \end{definition}

\begin{proposition} $M$ admits an $\a$-twisted $Spin^c$ structure
if and only if  there is a continuous map $\i: M\to X$ such that 
\[
\i^*([\a]) +  W_3(M)=0.
\]
 If $\i$ is an embedding, this is the  anomaly cancellation
condition obtained by Freed and Witten in \cite{FreWit}. 
\end{proposition} 
  
  \begin{proof} This is clear. \end{proof}
  
A morphism between   $\a$-twisted $Spin^c$  manifolds
$(M_1, \nu_1, \i_1, \eta_1)$ and $ (M_2, \nu_2, \i_2, \eta_2)$  is
a continuous map $f: M_1 \to M_2$ where  the following diagram 
 \ba\label{morphism}
   \xymatrix{
    M_1  \ar@/{}_{1pc}/[ddr]_{\i_1} \ar@/{}^{1pc}/[drr]^{\nu_1}
       \ar[dr]^{f }            \\
      & M_2 \ar[d]_{\i_2} \ar[r]^{\nu_2}   
                     & \BSO \ar@2{-->}[dl]_{\eta_2}\ar[d]^{W_3}       \\
      & X \ar[r]_\a   & K(\ZZ, 3)              }
   \na
 is a homotopy commutative diagram such that
 \begin{enumerate}
\item   $\nu_1$ is homotopic to $\nu_2 \circ f$ through a continuous  map $\nu: M_1 \times [0, 1] \to \BSO$;
\item  $\i_2 \circ f$ is homotopic to $\i_1$   through continuous   map $\i : M_1 \times [0, 1] \to X$;
\item   the composition of homotopies  $( \a \circ \i ) * (\eta_2 \circ (f\times Id) ) * (W_3 \circ \nu)$
 is homotopic to $\eta_1$.
   \end{enumerate} 
     Two  $\a$-twisted $Spin^c$  manifolds
$(M_1, \nu_1, \i_1, \eta_1)$ and $ (M_2, \nu_2, \i_2, \eta_2)$  are called isomorphic
if  there exists a diffeomorphism $f: M_1 \to M_2$ such that  the above holds. If the identity
map on $M$ induces an isomorphism between $(M, \nu_1, \i_1, \eta_1)$ and $ (M, \nu_2, \i_2, \eta_2)$, then these two $\a$-twisted $Spin^c$ structures are called equivalent.

Orientation reversal in the Grassmannian model defines an involution 
\[
r: \BSO \longrightarrow \BSO.
\]
 Choose a good cover $\{V_i\}$ of $M$ and hence a trivialisation of the universal bundle over $\BSO(n)$ with transition functions 
\[
g_{ij}: V_i\cap V_j \longrightarrow SO(n).
\]
Let  $\tilde{g}_{ij}: V_i\cap V_j \longrightarrow Spin^c(n)$ be   a lifting of $g_{ij}$.
 Then  $\{c_{ijk}\}$, obtained from
\[
\tilde{g}_{ij} \tilde{g}_{jk} = c_{ijk}\tilde{g}_{ik},
\]
defines $[W_3] \in H^3(\BSO, \ZZ)$.  Let $h$ be the diagonal matrix with the first $(n-1)$ diagonal entries  $1$ and the last entry $-1$. Then $\{h g_{ij} h^{-1}\}$ are the transition functions 
for the universal  bundle over $\BSO(n)$ with the opposite orientation.  Note that  $\{h \tilde{g}_{ij} h^{-1}\}$  is a lifting of $\{h g_{ij} h^{-1}\}$, which leaves $\{c_{ijk}\}$ unchanged.  We have 
$[W_3] =  [W_3\circ r] \in H^3(\BSO, \ZZ)$. 
Hence there is a   homotopy connecting $W_3$ and $W_3\circ r$.  (It is unique up to homotopy as $H^2(\BSO, \ZZ)=0$).  Given an $\a$-twisted $Spin^c$ manifold   $(M, \nu, \i, \eta)$, let $-M$ be the same manifold with the orientation reversed.  Then the homotopy
commutative diagram
\[
\xymatrix{M \ar[d]_{\i} \ar[r]^{\nu} & 
\BSO \ar[r]^r \ar[d]^{W_3}
 \ar@2{-->}[dl]_{\eta}& \BSO  \ar@2{-->}[dl] \ar@/{}^{1.7pc}/[dl]^{W_3} \\
X \ar[r]_\a  & K(\ZZ, 3)  &   } 
\]
determines a unique equivalence class of $\a$-twisted $Spin^c$ structure on $-M$, called the 
{\bf opposite}  $\a$-twisted $Spin^c$ structure,  simply denoted  by $-(M, \nu, \i, \eta)$.

\begin{definition}
 A   geometric cycle for $(X, \a)$  is 
a quintuple $(M, \i, \nu, \eta, [E])$ where $[E]$ is a K-class in $K^0(M)$ and 
 $M$ is a smooth closed manifold  equipped with an $\a$-twisted $Spin^c$ structure
 $(M, \i, \nu, \eta)$. 
  
Two geometric cycles $(M_1, \i_1, \nu_1, \eta_1, [E_1])$ and $ (M_2, \i,_2 \nu_2, \eta_2, [E_2])$
are isomorphic 
if there is an isomorphism $f:  (M_1, \i_1, \nu_1, \eta_1) \to  (M_2, \i_2,  \nu_2, \eta_2)$,
as $\a$-twisted $Spin^c$ manifolds over $X$,  such that $f_! ([E_1]) = [E_2]$.
\end{definition}

Let $\Gamma (X, \a)$ be the collection of all geometric cycles for $(X,  \a)$. We now impose an equivalence relation $\sim$ on $\Gamma (X, \a)$, generated by the following three elementary 
relations:
\begin{enumerate}
\item  {\bf Direct sum -  disjoint union}

\nin If  $(M , \i , \nu , \eta , [E_1])$ and $ (M , \i,  \nu , \eta , [E_2])$ 
are two geometric cycles with the same $\a$-twisted $Spin^c$ structure,
then 
\[
(M , \i , \nu , \eta , [E_1]) \cup  ( M , \i , \nu , \eta , [E_2]) \sim (M , \i , \nu , \eta , [E_1]+ [E_2]).
\]
\item  {\bf Bordism}

\nin Given  two geometric cycles $(M_1, \i_1, \nu_1, \eta_1, [E_1])$ and $ (M_2, \i_2, \nu_2, \eta_2, [E_2])$, if
there exists a $\a$-twisted $Spin^c$ manifold  $(W, \i, \nu, \eta)$ and $[E]\in K^0(W)$  such that 
\[
\pa (W, \i, \nu, \eta) =  
-(M_1, \i_1, \nu_1, \eta_1) \cup   (M_2, \i_2,  \nu_2, \eta_2)
\]
and $\pa ([E]) = [E_1] \cup [E_2]$. Here $-(M_1, \i_1, \nu_1, \eta_1)$
denotes  the manifold  $M_1$  with the  opposite $\a$-twisted $Spin^c$ structure.

\item   {\bf $Spin^c$ vector bundle modification}

\nin
 Suppose we are given a geometric cycle
 $(M, \i, \nu, \eta, [E]) $ and a  $Spin^c$ vector bundle $V$  over $M$ with 
  even dimensional fibers.  Denote by $\underline{\RR}$ the trivial rank one real 
  vector bundle. Choose a Riemannian metric on $V\oplus \underline{\RR}$, let
  $$\hat{M}= S(V\oplus \underline{\RR})$$  be the sphere bundle of $V\oplus \underline{\RR}$. Then
   the vertical tangent bundle $T^v(\hat{M})$ of $ \hat{M}$ admits a natural $Spin^c$ structure
   with an associated $\ZZ_2$-graded spinor bundle  $S^+_V\oplus S^-_V$ . Denote by
  $\rho: \hat{M} \to M$   the projection  which is   K-oriented.  Then
  \[
  (M, \i, \nu, \eta, [E]) \sim (\hat{M}, \i\circ \rho , \nu \circ \rho, \eta \circ \rho, [\rho^*E\otimes S^+_V]).
  \]
\end{enumerate}

\begin{definition} \label{twisted:geo} Denote by  $\hKg_*(X, \a) = \Gamma (X, \a)/\sim$ the 
geometric twisted $K$-homology. Addition is given by disjoint union - 
direct sum relation. Note that the equivalence relation $\sim$ preserves the parity
of the dimension of the underlying $\a$-twisted $Spin^c$ manifold. Let 
$\hKg_{0}(X, \a) $ (resp. $\hKg_1(X, \a)$  ) the subgroup of $ \hKg_*(X, \a)$
determined by all geometric cycles with even (resp. odd) dimensional
$\a$-twisted $Spin^c$ manifolds. 
\end{definition}

\begin{remark} \begin{enumerate}
 \item If  $M$, in  a geometric cycle $(M, \i, \nu, \eta, [E])$  for $(X, \a)$,  is a compact manifold with boundary, then $[E]$ has to be a class in $K^0(M, \pa M)$. 
 \item If $f: X\to Y$ is a continuous map and $\a: Y\to K(\ZZ, 3)$ is a twisting, then there is a natural
homomorphism  of abelian groups
\[
f_*:  \hKg_{ev/odd}(X, \a \circ f ) \longrightarrow \hKg_{ev/odd}(Y, \a) 
\]
sending $[M, \i, \nu, \eta,  E ]$ to $[M,f \circ  \i ,  \nu, \eta,  E]$.   
\item Let $A$ be a closed subspace of $X$, and $\a$ be a twisting on $X$.   A relative   geometric cycle for $(X, A;  \a)$  is  a 
quintuple   $(M, \i, \nu, \eta, [E])$    such that
\begin{enumerate}
\item $M$ is a smooth  manifold  (possibly with boundary), equipped  with an $\a$-twisted $Spin^c$ structure $(M, \i, \nu, \eta)$;
\item if $M$ has a non-empty boundary, then  $\i (\pa M) \subset  A$;
\item $[E]$ is a K-class in $K^0(M)$ represented by a $\ZZ_2$-graded vector bundle $E$ over $M$, or
a continuous map  $M \to BU(\infty)$. 
\end{enumerate}
\end{enumerate}
\end{remark}
The  relation $\sim$ generated by disjoint union - direct sum,  bordism
and $Spin^c$ vector bundle
modification is an equivalence relation.  The collection of relative geometric cycles, modulo
the equivalence relation   is denoted by
\[\hKg_{ev/odd}(X, A; \a ).
\]

 There exists  a natural  homomorphism, called the assembly map
$$\mu: \hKg_{ev/odd}(X, \a) \to \hKa_{ev/odd} (X, \a)$$
whose definition (which we will now explain) requires a careful study of geometric cycles.

Given a geometric cycle $(M, \i, \nu, \eta, [E])$,  equip $M$ with a Riemannian metric.   Denote by  $\CL (TM)$ the bundle of complex Clifford algebras of $TM$ over $M$.  The algebra of sections,   $C(M, \CL (TM)) $,    is Morita equivalent to 
 $C(M, \tau^* \BSpin^c (\cK) )$. Hence, we have   a canonical isomorphism
 \[ 
 \hKa_{ev/odd} (M, W_3\circ \tau ) \cong  KK^{ev/odd} ( C(M, \CL (M) ), \CC) 
 \]
 with the degree shift by $\dim M  (mod \ 2)$. 
Applying  Kasparov's Poincar\'{e} duality  (Cf. \cite{Kas1}) 
\[
 KK^{ev/odd} (\CC, C(M)) \cong KK^{ev/odd} (C(M, \CL (M) ), \CC), 
 \]
we obtain  a canonical isomorphism 
 \[
PD:  K^0(M)  \cong \hKa_{ev/odd} (M,  o_M ), 
\]
with the degree shift by $\dim M  (mod \ 2)$. 
 The fundamental class   $[M] \in \hKa_{ev/odd} (M, o_M )$ is 
the Poincar\'{e} dual of the unit element in $K^0(M)$. Note that
 $[M]\in \hKa_{ev}(M, o_M))$ if $M$ is even dimensional and $[M]\in \hKa_{odd}(M, o_M)$ if $M$ is odd dimensional. The cap product 
 \[\cap: 
 \hKa_{ev/odd} (M,  o_M ) \otimes K^0(M)  \longrightarrow \hKa_{ev/odd} (M, o_M )  
 \]
 is  defined by the Kasparov product.  We remark that Poincar\'e duality   is given by  the cap product of the fundamental $K$-homology class $[M]$ 
 \[
 [M] \cap: K^0(M)  \cong \hKa_{ev/odd} (M,  o_M ). 
 \]
 
Choose  an embedding $i_k: M \to \RR^{n+k}$ and take the resulting normal bundle $\nu_M$. The natural isomorphism
   \[
   TM \oplus \nu_M \oplus \nu_M \cong \underline{\RR}^{n+k} \oplus \nu_M
   \]
   and the canonical $Spin^c$ structure on $\nu_M \oplus \nu_M$ define a canonical homotopy between 
 the orientation twisting  $o_M$  of $TM$ and the orientation twisting $o_{\nu_M}$ of $\nu_M$.  This  canonical homotopy defines an isomorphism
 \ba\label{1}
I_*: \hKa_{ev/odd} (M, o_M ) \cong \hKa_{ev/odd}(M,     o_{\nu_M} ) .
\na

  Given an    $\a$-twisted $Spin^c$  manifold $(M, \nu,  \i, \eta)$   over $X$,  the homotopy $\eta$ induces an isomorphism
$\nu^* \BSpin^c   \cong   \i^* \cP_\a $ as principal $K(\ZZ, 2)$-bundles  on $M$. Hence there is   an isomorphism 
\[\xymatrix{
   \nu^* \BSpin^c (\cK)   \ar[rr]^{  \eta^*}_{ \cong }&&  \i^* \cP_\a  (\cK) }
\]
as bundles of $C^*$-algebras  on  $M$.    This isomorphism
determines a canonical isomorphism 
between the corresponding continuous trace $C^*$-algebras 
\[
C(M, \nu^* \BSpin^c (\cK))  \cong  C(M,  \i^* \cP_\a(\cK) ).
\]
 Hence,  we have a canonical isomorphism
\ba\label{2}
\eta_*:   \hKa_{ev/odd} (M, o_{\nu_M}  )  \cong   \hKa_{ev/odd} (M,      \a \circ \i  ). 
\na

Now we can define the assembly map as 
\[ 
\mu  (M, \i, \nu, \eta, [E]) =   \i_*\circ \eta_* \circ I_* ([M]\cap [E]) 
\]
in $\hKa_{ev/odd} (X, \a)$. Here $\i_*$ is  the natural push-forward map in analytic twisted  $K$-homology. 

\begin{theorem} \label{geo:ana}  (Theorem 6.4 in \cite{Wang}) The 
 assembly map $\mu: \hKg_{ev/odd}(X, \a) \to \hKa_{ev/odd} (X, \a)$
  is an isomorphism for any {\bf smooth}  manifold $X$ with a twisting $\a: X \to K(\ZZ, 3)$.
\end{theorem} 

The proof follows by establishing the existence of   a natural map  $\Psi: \hKt_{ev} (X , \a ) \to \hKg_0(X, \a)$ such that 
   the following  diagram 
  \[
  \xymatrix{ 
& \hKt_{ev/odd} (X , \a )  \ar[dl]_{\Psi} \ar[dr]^{\Phi}_{\cong}
&\\
\hKg_{ev/odd}(X, \a ) \ar[rr]^{\mu} & & \hKa_{ev/odd} (X, \a)}
\]
commutes.  All the maps in the diagram are isomorphisms.  

\begin{remark} This theorem is generalised in \cite{BW} to the case of
 any CW pair $(X, A)$. That is, it is shown that the equivalence between 
the geometric twisted $K$-homology and the analytic twisted  $K$-theory holds in this
more general situation.
\end{remark}

\begin{corollary} $ \hKa_{ev/odd} (X, \a) \cong  \hKa_{ev/odd} (X, - \a)$.
\end{corollary} 
\begin{proof} By the Brown representation theorem (\cite{Swi}), there is a continuous map $\hat{i}: K(\ZZ, 3) \to K(\ZZ, 3)$ (unique up to homotopy as $H^2(K(\ZZ, 3), \ZZ)=0$)    such that   
\[
[\hat{i} \circ \a ] = -[\a]  \in H^3(X, \ZZ)
\]
for any map $\a: X\to K(\ZZ, 3)$.  Then we  have
\[
[\hat{i} \circ W_3 ] = -[W_3]  \in H^3(BSO, \ZZ).
\]
As $[W_3]$ is  2-torsion, we  know that $[\hat{i} \circ W_3 ] = -[W_3] = [W_3]$. Therefore, there
is a homotopy  $\eta_0$   connecting $\hat{i} \circ W_3 $ and $W_3$, that is, 
 the following diagram is homotopy  commutative 
\[
\xymatrix{
\BSO  \ar[d]^{W_3} \ar@/{}^{1.2pc}/[drr]^{W_3} && \\
K(\ZZ, 3) \ar@{<==}[ur]_{\eta_0} \ar[rr]^{\hat{i}}&&  K(\ZZ, 3).}
\] 
Note that the homotopy  class of $\eta_0$ as a homotopy connecting $W_3$ and 
$\hat{i} \circ W_3$  is unique due to the fact that $H^2(BSO, \ZZ)=0$. 

Given an $\a$-twisted $Spin^c$  manifold $(M, \i, \nu, \eta)$, then the following homotopy commutative diagram
\[
\xymatrix{M \ar[d]_{\i} \ar[r]^{\nu} & 
\BSO  \ar[d]^{W_3} \ar@2{-->}[dl]_{\eta}  \ar@/{}^{1.3pc}/[drr]^{W_3}  &&  
  &  \\
X \ar[r]_\a  & K(\ZZ, 3) \ar@{<==}[ur]_{\eta_0} \ar[rr]^{\hat{i}} &  & K(\ZZ, 3)  } 
\] 
defines a unique (due to $H^2(\BSO, \ZZ) =0$)  equivalence class of $(-\a)$-twisted   $Spin^c$ structures. Here $-\a =\hat{i}\circ \a$.
We denote by $\hat{i}(M, \i, \nu, \eta)$ this $(-\a)$-twisted   $Spin^c$ manifold. Obviously, 
\[
\hat{i}\bigl(\hat{i}(M, \i, \nu, \eta) \bigr) = (M, \i, \nu, \eta).
\]
The isomorphism $ \hKa_{ev/odd} (X, \a) \cong  \hKa_{ev/odd} (X, - \a)$ is induced by the involution
$\hat{i}$ on geometric cycles. 
\end{proof}

\section{The Chern character in twisted  $K$-theory}

In this Section, we will review the Chern character map in twisted  $K$-theory on smooth  manifolds developed in \cite{CMW} using gerbe connections and curvings. For the topological and analytic definitions, see \cite{AS2} and
\cite{MatSte} respectively.   Recently, Gomi and Terashima  in \cite{GoTe}
gave  another  construction of a Chern
character for twisted $K$-theory using a notion of connection on a finite-dimensional 
approximation of a twisted family of Fredholm operators developed by Gomi (\cite{Gomi}.

 \subsection{Twisted Chern character}

For a fibration $\pi^*: Y \to X$,  let $Y^{[p]}$ denote the $p$th fibered
product. There are projection maps $\pi_i \colon Y^{[p]} \to
Y^{[p-1]}$ which omit the $i$th element for each $i = 1 \dots
p$.  These define a map
\begin{equation}\label{definingdelta1}
\delta \colon \Omega^q(Y^{[p-1]}) \to \Omega^q(Y^{[p]})
\end{equation}
by
\begin{equation}\label{definingdelta2}
\delta(\omega) = \sum_{i=1}^p (-1)^i \pi_i^*(\omega).
\end{equation}
Clearly $\delta^2 = 0$. In fact, the $\delta$-cohomology of this complex vanishes identically, hence, the sequence 
\[\xymatrix{
0\ar[r]& \Omega^q(X) \ar[r]^{\pi^*} & \Omega^q(Y) \ar[r]^\delta  \cdots  &   \Omega^q(Y^{[p-1]}) \ar[r]^\delta  & \Omega^q(Y^{[p]}) \ar[r] &\cdots }
\]
is exact.

Returning now to our particular example, a bundle gerbe  connection on $\cP_\a$ is a unitary connection
$\theta$ on the principal $U(1)$-bundle $\cG_\a$ over 
$\cP_\a^{[2]}$ which  commutes with the bundle gerbe product.    
A bundle gerbe connection $\theta$ has curvature
\[
F_\theta \in \Omega^2(\cP_\a^{[2]})
\]
 satisfying $\delta  (F_\theta) = 0$. 
  There exists a two-form $\omega$  on  $\cP_{\a}$ such that
\[
F_\theta = \pi_2^* (\omega) -\pi_1^* (\omega).
\]
Such an $\omega$  is called a curving for the gerbe connection $\theta$. The choice of a curving is not unique,
  the ambiguity in the choice is precisely the addition of the pull-back 
  to $\cP_{\a}$ of
a two-form on  $X$.   Given a choice of curving $\omega$,  there is   a unique closed  three-form on $\beta $ on $X$ satisfying $d\omega  = \pi^*  \beta $.    We  denote by 
\[
\check{\a} = (\cG_\a, \theta, \omega)
\]
 the lifting bundle  gerbe $\cG_\a$ with the connection $\theta$ and  a curving $\omega$. 
 Moreover
 $H= \dfrac{\beta}{2\pi \sqrt{-1}}$  is a de Rham representative for
the Dixmier-Douady class $[\a]$. We shall 
 call   $\check{\alpha}$ the {\bf differential twisting}, as it is the twisting in differential twisted  $K$-theory (Cf. \cite{CMW}). 
 
 The following theorem is established in \cite{CMW}.

\begin{theorem} Let $X$ be  a   smooth  manifold, $\pi: \cP_\a \to X$ be  a  principal
$PU(\cH)$ bundle over $X$ whose  classifying map is given by 
$\alpha: X\longrightarrow K(\ZZ, 3)$.  Let $  \check{\a} = (\cG_\a, \theta, \omega) $ be a bundle gerbe
  connection $\theta$  and a  curving $\omega$ on the   lifting bundle gerbe $\cG_\a$. There
  is a well-defined  twisted   Chern character  
 \[
 Ch_{\check{\a}}: K^* (X,  \a ) \longrightarrow
 H^{ev/odd}(X,  d - H ). 
\]
Here the groups $H^{ev/odd}(X,  d - H )$ are  the twisted  cohomology groups of the complex of differential 
forms on X with the coboundary operator given by $d-H$. The  twisted   Chern character 
is functorial under the pull-back.  Moreover,  given another differential twisting 
$\check{\a} + b = (\cG_\a, \theta, \omega +\pi^*b) $
for a 2-form $b$ on $X$,  
\[
Ch_{\check{\a}+b}  = Ch_{\check{\a}} \cdot \exp (\dfrac{b}{2\pi\sqrt{-1}}). 
\]
  \end{theorem}
\begin{proof}

Choose a good open cover $  \{V_i\}$ of $X$
such that $\cP_\a \to X$ has  trivializing sections $\phi_i$ over each $V_i$ with transition functions
$g_{ij}:  V_i \cap V_j \longrightarrow PU(\cH)$
satisfying  $\phi_j = \phi_i g_{ij}$. 
Define $\{\sigma_{ijk}\}$ by   $\hat{g}_{ij} \hat{g}_{jk}= \hat{g}_{ik}\sigma_{ijk} $ for 
 a lift of $g_{ij}$ to $\hat{g}_{ij}:  V_i \cap V_j \to U (\cH) $.
 Note that the pair $(\phi_i, \phi_j)$  defines a section of $\cP^{[2]}_{\a}$
over $V_i \cap V_j$. 
  The connection $\theta$ can be pulled back by $(\phi_i, \phi_j)$
to define  a 1-form
$A_{ij}$ on $V_i\cap  V_j$ and the curving $\omega$ can be pulled-back by the $\phi_i$
 to define two-forms $B_i$ on $V_i$. Then the differential twisting defines 
the triple
\ba\label{deligne:cocycle}
\{(\sigma_{ijk}, A_{ij}, B_i)\}
\na
which  is a   degree two smooth Deligne cocycle. 
Now we explain in some detail the twisted Chern characters in both the odd
and even case following \cite{CMW}.
  
\underline{\bf The even case:}  
As a model for the $K^0$ classifying space,  
we choose $\Fred$,  the space of bounded self-adjoint  Fredholm operators with 
essential spectrum $\{\pm 1\}$ and otherwise discrete spectra, with a grading operator $\Gamma$ which anticommutes with the 
given family of Fredholm operators.

A  twisted K-class in $K^0(X,  \a ) $ can be represented by  $f: \cP_\a \to \Fred $, a $PU(\cH)$-equivariant  family of Fredholm operators.  We can select an open cover $\{V_i\}$  of $X$ such that on each $V_i$  there is a local section  $\phi_i: V_i \to \cP_\a$ and  for each $i$ the Fredholm operators $f(\phi_i(x))$, $x\in V_i$ have a 
gap in the spectrum at  both $\pm \lambda_i \neq 0.$ Then over $V_i$ we have a finite rank vector bundle $E_i$ defined  by the spectral projections of the operators  $f(\phi_i(x))$
corresponding to the interval  $[-\lambda_i,\lambda]$.

 Passing to a finer cover $\{U_i\}$ if necessary, we may assume that $E_i$ is a trivial vector bundle over $U_i$  of rank $n_i.$
 Choosing a trivialization of $E_i$  gives  a  $\ZZ_2 $ graded parametrix  $q_i$ (an inverse up to
finite rank operators) of the family $f\circ \phi_i.$    In the index zero sector the operator $q_i(x)^{-1}$ is defined as the direct sum of the restriction
of $f(\phi_i(x))$ to the orthogonal complement of $E_i$ in $\cH$ and an isomorphism between the vector bundles $E^+_i$ and $E_i^-.$  Clearly
then $f(\phi_i(x)) q_i(x) =1$ modulo rank $n_i$ operators. 
In the case of nonzero index one defines a parametrix as a graded invertible operator $q_i$ such that $f(\phi_i(x)) q_i(x) = s_n$ modulo
finite rank operators, 
with $s_n$ a fixed Fredholm operator of  index $n$ equal to the index of $f(\phi_i(x)).$ 

On the overlap $U_{ij}$ we have a pair of parametrices $q_i$ and $ q_j $ of families of $f\circ \phi_i$
and $f\circ \phi_j$  respectively. These are related by an invertible operator $f_{ij}$ 
which is of the form
$1+$ a finite rank operator, 
$$  \hat{g}_{ij} q_j (x) \hat{g}_{ij}^{-1}  = q_i(x) f_{ij}(x).$$ 
The conjugation on the left hand side  by $\hat{g}_{ij}$  comes from the equivariance relation
\[
 f(\phi_j(x)) = f(\phi_i(x)g_{ij}(x)) = \hat{g}_{ij}(x)^{-1} f(\phi_i(x))\hat{g}_{ij}(x).
\] 

The system $\{f_{ij}\}$ does not quite satisfy the \v{C}ech cocycle relation needed to define a principal bundle, because of the
different local sections $\phi_i: U_i \to \mathcal{P}_{\sigma}$ involved. Instead, we have
 on $U_{ijk}$ 
$$\hat{g}_{jk} q_k \hat{g}_{jk}^{-1} = q_j f_{jk} = (\hat{g}_{ij}^{-1} q_i f_{ij} \hat{g}_{ij} )f_{jk}
= \hat{g}_{jk}( \hat{g}_{ik}^{-1} q_i f_{ik}\hat{g}_{ik} ) \hat{g}_{jk} ^{-1}.$$
Using the   relation $\hat{g}_{jk}  \hat{g}_{ik}^{-1} =\sigma_{ijk} \hat{g}_{ij}^{-1}$,  we get
\[
    \hat{g}_{jk} ( \hat{g}_{ik}^{-1} q_i f_{ik}\hat{g}_{ik} )  \hat{g}_{jk}^{-1} = \hat{g}_{ij}^{-1} q_i f_{ik} \hat{g}_{ij} 
\]
multiplying the last equation from right by $\hat{g}_{ij}^{-1}$ and from the 
left by $q_i^{-1}\hat{g}_{ij} $ one gets the twisted cocycle relation
\[
 f_{ij}(\hat{g}_{ij} f_{jk} \hat{g}_{ij}^{-1}) = f_{ik}, 
\]
 which is independent of the choice of the lifting $ \hat{g}_{ij}$.  For simplicity, we will just
 write the above  twisted cocycle relation as 
\ba\label{twisted:cocycle}
 f_{ij}(g_{ij} f_{jk} g _{ij}^{-1}) = f_{ik}. 
 \na

This twisted cocycle relation (\ref{twisted:cocycle})    actually  defines an untwisted cocycle relation
for $\{(g_{ij}, f_{ij})\}$  in the twisted product 
\[
\mathfrak{G} = PU(\cH) \rtimes GL(\infty),
\] 
where the group $PU(\cH)$ acts on the group $GL(\infty)$ of invertible  $1+$ finite rank operators by conjugation.  Thus the product
in $\mathfrak{G}$ is given by
$$ (g, f) \cdot (g',f')= (gg', f(g f' g^{-1}) ).$$
The cocycle relation for the pairs $\{(g_{ij}, f_{ij})\}$ then encodes both the cocycle relation for the transition functions $\{g_{ij}\}$ of the
$PU(\cH)$ bundle over $X$ and the twisted cocycle relation  (\ref{twisted:cocycle}).  In summary, this cocycle   $\{(g_{ij}, f_{ij})\} $ defines a principal
$\mathfrak{G}$ bundle over $X.$

   The classifying  space $B\mathfrak{G}$ is a fiber bundle over $K(\mathbb Z,3).$ 
  The fiber   at  each point in $ K( \mathbb Z,3)$ is 
  homeomorphic (but not canonically so) to the space $\Fred$ of graded Fredholm operators;  to set up the isomorphism one needs a choice of element  in  each fiber.    Given a principal $PU(\cH)$-bundle $\cP_\a$
 over $X$ defined by $\a: X \to K(\ZZ, 3)$,  the  even twisted  $K$-theory  $K^0(X, \cP_{\a})$ is the
set of homotopy classes of maps $X \to B\mathfrak{G}$ covering the map $\a$.

Next we construct the twisted Chern character from a connection $\nabla$ on a principal $\mathfrak{G}$ bundle  over $X$ associated to the cocycle $\{(g_{ij}, f_{ij})\}$.   Locally, on a 
good open cover  $\{U_i\}$ of $X$ we can lift the connection to a connection taking values in the Lie algebra $\hat{\mathbf{g}}$ of the central
extension $U(H) \times GL(\infty)$ of $\mathfrak{G}.$  Denote by $\hat F_{\nabla}$ the curvature of this connection. On the overlaps $\{U_{ij}\}$ 
the curvature satisfies a twisted relation
$$ \hat F_{\nabla,j} =  Ad_{(g_{ij}, f_{ij})^{-1}}  \hat{F}_{\nabla,i}  + g^*_{ij} c,$$
where $c$ is the curvature of the canonical connection $\theta$ on
 the principal $U(1)$-bundle $U(\cH)\to PU(\cH).$ 

Since the Lie algebra $\bold{u}(\infty) \oplus \CC$ is an ideal in the Lie algebra of $U(\cH) \rtimes GL(\infty),$ the projection $F'_{\nabla,i}$ of the curvature
$\hat F_{\nabla,i}$ onto this subalgebra transforms in the same way as $\hat F$ under change of local trivialization. It follows that for a $PU(\cH)$-equivariant map $f: \cP_\a \to \Fred$,  we can define a twisted 
Chern character form of $f$ as
\ba\label{local}
 ch_{\check{\a}} (f, \nabla )   = e^{B_i} \text{tr}\, e^{ F'_{\nabla,i}/2\pi i},
 \na
 over $V_i$. 
Here the trace is well-defined on $\mathbf{gl}(\infty)$ and on the center $\CC$ it is defined as the coefficient of the unit operator.  
Note that   $ch_{\check{\a}} (f, \nabla )$  is globally defined and $(d-H)$-closed
$$(d-H) ch_{\check{\a}} (f,  \nabla )  =0,$$ 
and depends on the differential twisting
\[
\check{\a} = (\cG_\a, \theta, \omega ).
\]
Let $f_0$ and $f_1$ be homotopic,  and $\nabla_0$ and $\nabla_1$ be two connections on the  principal $\mathfrak{G}$ bundle   over $X$, we have a Chern-Simons type form
\[
CS((f_0, \nabla_0), (f_1, \nabla_1)), 
\]
well-defined modulo $(d-H)$-exact forms, such that
\ba\label{CS}
 ch_{\check{\a}} (f_1, \nabla_1 )  -  ch_{\check{\a}} (f_0, \nabla_0 )  = (d-H) CS((f_0, \nabla_0), (f_1, \nabla_1)).
 \na
The proof follows directly from the local computation  using (\ref{local}).
Hence, the $(d-H)$-cohomology class of $ ch_{\check{\a}} (f,  \nabla ) $ does not depend   choices of  a connection $\nabla$ on the  principal $\mathfrak{G}$ bundle   over $X$,  and depends only on the homotopy class of $f$.    We denote the $(d-H)$-cohomology class of $ch_{\check{\a}} (f,  \nabla ) $ by
$Ch_{\check{\a}} ([f_1])$ which is a natural homomorphism
\[
 Ch_{\check{\a}}: K^0 (X,  \a ) \longrightarrow
 H^{ev}(X,  d - H ). 
\]
 From (\ref{local}), we have 
 \[
Ch_{\check{\a}+b}  = Ch_{\check{\a}} \cdot \exp (\dfrac{b}{2\pi\sqrt{-1}}). 
\]
for a  differential twisting 
$\check{\a} + b = (\cG_\a, \theta, \omega +\pi^*b) $.

\underline{\bf The odd case:}  The odd case is a little easier. First, as a model for the $K^1$ classifying space,  
we choose $U(\infty) =\vlim_{n} U(n)$, the stabilized unitary group.  Let  $\Theta$ be the universal odd 
character form on $U(\infty)$ defined by the canonical left invariant $u(\infty)$-valued form on $U(\infty)$.

  Let $\cH = \cH_+ \oplus \cH_-$ be a polarized 
Hilbert space and let  $U_{res}=U_{res}(\cH)$ denotes the group of unitary operators in $\cH$ with Hilbert-Schmidt off-diagonal blocks.  The conjugation action of $U(\cH_+)\times U(\cH_-)$  on $U_{res}$ defines an action of $PU_0(\cH) = P(U(\cH_+)\times U(\cH_-))$  on $U_{res}$. 
Note that the classifying space of 
 $U_{res}$ is $U(\infty)$.     

Define
\[
\mathfrak{H} =PU_0(\cH) \rtimes U_{res}.
\]
Then given a principal $PU_0(\cH)$-bundle $\cP_\a$
 over $X$ defined by $\a: X \to K(\ZZ, 3)$,  the odd  twisted  $K$-theory  $K^1(X, \a)$ is the set of homotopy classes of maps $X \to B\mathfrak{H}$ covering the map $\a$. These are represented by  $PU_0(\cH)$-equivariant maps $f:  \cP_\a\to U(\infty)$.   With respect to trivializing sections $\phi_i$ over each 
 $V_i$.  Then
 \[
 e^{B_i}  (f\circ \phi_i)^*\Theta
 \]
 is a globally defined and $(d-H)$-closed differential form on $X$.  This defines the odd
 version of the twisted Chern character
 \[
 Ch_{\check{\a}}:  K^1(X, \a) \longrightarrow H^{odd}(X, d-H).
 \]
 \end{proof}

   \subsection{Differential twisted  $K$-theory}
   
     Recall  that  the Bockstein exact sequence in complex  $K$-theory for any finite CW complex:
   \ba \label{Bockstein}
 \xymatrix{
 K^0(X)\ar[r]^{ch}& H^{ev}(X, \RR) \ar[r]& K^0_{\RR/\ZZ} (X)
 \ar[d]\\
 K^1_{\RR/\ZZ} (X) \ar[u] & H^{odd} (X, \RR)   \ar[l]&
   K^1(X) \ar[l]^{ch}
   }
   \na 
   where $K^*_{\RR/\ZZ}(X)$ is $K$-theory with $\RR/\ZZ$-coefficients as
   in \cite{Kar1} and \cite{Ba}.
   
   %   \[\xymatrix{ \text{homotopy fiber}\ar[r] & BU \ar[r]^{ch \qquad } &  \prod_{k\geq 0} K(\RR, 2k).   \]

Analogously, in  twisted  $K$-theory, given a smooth manifold $X$ with a twisting $\a: X\to K(\ZZ, 3)$, upon a choice of a differential twisting
   \[
   \check{\a} = (\cG_\a, \theta, \omega)
   \]
   lifting $\a$, we have  the corresponding Bockstein exact sequence in twisted  $K$-theory
    \ba \label{Bockstein:twisted}
 \xymatrix{
 K^0(X, \a)\ar[r]^{Ch_{\check{\a}} \qquad }& H^{ev}(X, d-H ) \ar[r]& K^0_{\RR/\ZZ} (X, \a)
 \ar[d]\\
 K^1_{\RR/\ZZ} (X, \a) \ar[u] & H^{odd} (X, d-H)   \ar[l]&
   K^1(X, \a) \ar[l]^{\qquad Ch_{\check{\a}}}
   }.
   \na 
   Here $ K^0_{\RR/\ZZ} (X, \a)$ and $ K^1_{\RR/\ZZ} (X, \a)$ are subgroups of
   differential twisted  $K$-theory, respectively $ \check{ K}^0(X,\check{ \a} )$ and $ \check{ K}^1(X,\check{ \a} )$ (see \cite{CMW} for the detailed construction).  Here we give another equivalent
   construction of differential twisted  $K$-theory.
   
  Fix a  choice of  a connection $\nabla$ on a principal $\mathfrak{G}$ bundle   over $X$. 
 Then  $ \check{ K}^0(X,\check{ \a} )$ is the abelian group generated by pairs
  \[
  \{(f, \eta)\},
  \]
 modulo  an equivalence relation, where $ f: \cP_\a \to \Fred$ is a $PU(\cH)$-equivariant map  and 
 $\eta$ is an odd differential form modulo $(d-H)$-exact forms.  Two pairs $(f_0, \eta_0)$ and 
 $ (f_1, \eta_1) $ are called equivalent if and only if
 \[
 \eta_1 -\eta_0 = CS((f_1, \nabla), (f_0, \nabla)). 
 \]
 The differential Chern character form of $f$ is given by  
 \[
 ch_{\check{\a}} (f,  \nabla) - (d-H) \eta
 \]
which defines a homomorphism
\[
ch_{\check{\a}}:   \check{ K}^0(X,\check{ \a} )  \longrightarrow \Omega^{ev}_0(X, d-H), 
\]
where $ \Omega^{ev}_0(X, d-H)$  is the image of $ch_{\check{\a}}:  \check{ K}^0(X,\check{ \a} )  
\to \Omega^{ev}(X)$.  The kernel of $ch_{\check{\a}}$ is  isomorphic to $ K^1_{\RR/\ZZ} (X, \a)$.

Similarly, we define the odd differential twisted  $K$-theory $ \check{ K}^1(X,\check{ \a} )$  with the 
differential Chern character form homomorphism
\[
ch_{\check{\a}}:   \check{ K}^1(X,\check{ \a} )  \longrightarrow \Omega^{odd}_0(X, d-H). 
\]
The  kernel of $ch_{\check{\a}}$ is  isomorphic to $ K^0_{\RR/\ZZ} (X, \a)$.  
 The following commutative diagrams were established in \cite{CMW} relating   differential twisted  $K$-theory with twisted  $K$-theory and with the diagram 
(\ref{Bockstein:twisted})
\[
  \xymatrix{  &   0\ar[d] & \\
    H^{odd}(X, d- H ) \ar[r]   \ar[d]   & K^1_{ \RR/\ZZ} (X, \check{\sigma}) \ar[rd]  \ar[d]  &    \\
    0\to\disp{\frac{ \Omega^{odd}(X)}{\Omega_0^{odd}(X,d-  H)}}   \ar[r]\ar[dr]_{d-H} &
     \check{K}^0  (X, \check{\sigma}) \ar[r]\ar[d]^{\ ch_{\check{\sigma}}} 
      & K^0(X, \sigma)\ar[d]^{Ch_{\check{\sigma}}} \to 0 \\
         & \Omega_0^{ev}(X, d- H) \ar[r] \ar[d]  & H^{ev}(X, d- H ) \\
         &   0  & }
         \]
         and 
     \[     \xymatrix{&   0\ar[d] & \\
       H^{ev}(X, d- H )  \ar[r] \ar[d] & K^0_{ \RR/\ZZ} (X, \check{\sigma} ) \ar[rd] \ar[d]  &   \\
    0\to\disp{\frac{ \Omega^{ev}(X)}{\Omega_0^{ev}(X, d-H)}}   \ar[r] \ar[r]\ar[dr]_{d-H} &
     \check{K}^{1}  (X, \check{\sigma}) \ar[r]\ar[d]^{\ ch_{\check{\sigma}}} 
      & K^1(X, \sigma)\ar[d]^{Ch_{\check{\sigma}}} \to 0 \\
         & \Omega_0^{odd}(X, d- H) \ar[r] \ar[d]  & H^{odd}(X, d- H ) \\
         &   0  &  }
         \]
         with exact horizontal and vertial sequences, and exact upper-right and exact lower-left 4-term
         sequences. 
         We expect that  these two commutative diagrams     uniquely characterize differential
         twisted  $K$-theory.

\subsection{Twisted Chern character for torsion twistings}

In this paper, we will only use the twisted Chern character for a torsion  twisting and in this case we will give an explicit
construction.
 Let $E$ be a real oriented vector bundle of   rank $2k$  over $X$ with its orientation twisting denoted
 by
 \[
 o(E):  X\to K(\ZZ, 3).
 \]
 The  associated  lifting bundle $\cG_{o(E)}$ has a canonical reduction to the
 $Spin^c$  bundle gerbe  $\cG_{W_3(E)}$.
 
 Choose a local trivialization of $E$ over a good open cover $\{V_i\}$ of $X$. Then the transition
functions 
\[
g_{ij}:  V_i\cap V_j \longrightarrow SO(2k).
\]
define an element in $H^1(X, \underline{SO(2k)})$ whose image under the Bockstein  exact sequence
\[
H^1(X, \underline{Spin(2k)}) \to H^1(X, \underline{SO(2k)}) \to H^2(X, \ZZ_2)
\]
is the second Stieffel-Whitney class $w_2(E)$ of $E$. 
Denote the differential twisting  by 
\[
\check{w}_2(E) = (\cG_{W_3(E)}, \theta, 0), 
\]
 the $Spin^c$  bundle gerbe   $\cG_{W_3(E)}$
with a flat connection  $\theta$ and a trivial curving.   With respect to a good cover $\{V_i\}$
of $X$ the  differential twisting $\check{w}_2(E)$ defines a Deligne
cocycle 
$\{(\a_{ijk}, 0, 0)\}$
with  trivial local $B$-fields, here $\a_{ijk}  = \hat{g}_{ij} \hat{g}_{jk}\hat{g}_{ki}  $ where $\hat{g}_{ij}: U_{ij} \to Spin^c (2n)$  is  a lift of $g_{ij}$.

By Proposition \ref{Cliff}, a    twisted K-class in $K^0(X,  o(E))$   can be represented by
a Clifford bundle, denoted $\cE$.   Equip  $\cE$  with a Clifford connection, 
 and  $E$  with a $SO(2k)$-connection. 
Locally, over each $V_i$  we let
$\cE|_{V_i} \cong S_i\otimes \cE_i $
where $S_i$ is the local fundamental spinor bundle associated to $E|_{V_i}$ with the standard Clifford  action
of $\Cliff (E|_{V_i})$ obtained from the fundamental representation of $Spin(2k)$.  Then $\cE_i$ is a complex vector bundle
over $V_i$ with a connection $\nabla_i$  such that
on $V_i\cap V_j$ 
\[
Ch(\cE_i, \nabla_i)  = Ch(\cE_j, \nabla_j). 
\]
Hence, the twisted Chern character 
\[
Ch_{\check{w}_2(E)}: K^0(X, o(E)) \longrightarrow  H^{ev} (X) 
\]
is given by   $[\cE]\mapsto \{[ch(\cE_i, \nabla_i)] = ch(\cE_i)\}$.  The proof of the following proposition is
straightforward. 

\begin{proposition}\label{ring:homo}
The   twisted Chern character 
 satisfies the following identities\begin{enumerate}
\item $Ch_{\check{w}_2(E_1\oplus E_2)} ([\cE_1\oplus \cE_2]) = 
Ch_{\check{w}_2(E_1)} ([\cE_1])  + Ch_{\check{w}_2(E_2)} ([\cE_2])$.
\item $Ch_{\check{w}_2(E_1\otimes  E_2)} ([\cE_1\otimes  \cE_2]) = 
Ch_{\check{w}_2(E_1)} ([\cE_1])   Ch_{\check{w}_2(E_2)} ([\cE_2])$. 
\end{enumerate}
\end{proposition} 

In the case that $E$ has a   $Spin^c$ structure
whose determinant line bundle is $L$,  there is a canonical
isomorphism 
\[
K^0(X) \longrightarrow  K^0(X, o(E)), 
\]
given by $[V] \mapsto [V\otimes S_E]$ where $S_E$ is the associated spinor bundle of $E$. Then we have
\[
Ch_{\check{w}_2(E)} ([V\otimes S_E]) =  e^{\frac{c_1(L) }{2}} ch ([V]),
\]
where $ch([V])$ is the ordinary Chern character of $[V] \in K^0(X)$. 

In particular, when $X$ is an even dimensional
Riemannian manifold, and $TX$ is equipped with the Levi-Civita
connection,  under the  identification of  $K^0(X,  o(E))$ with  the 
Grothendieck group of Clifford modules.  Then  
\ba\label{TX}
Ch_{\check{w}_2(X)} ([\cE]) =  ch(\cE/S)
\na
where $ch(\cE/S)$ is the relative  Chern character of the Clifford module $\cE$ constructed
in Section 4.1 of  \cite{BGV}.

  \section{Thom classes and Riemann-Roch  formula  in twisted  $K$-theory}

\subsection{The Thom class}
Given any   oriented  real vector bundle $\pi: E\to X$ of rank $2k$, $E$ admits a $Spin^c$ structure if 
its classifying map $\tau: X\to BSO(2k)$ admits a lift $\tilde{\tau}$
\[
\xymatrix{ & \BSpin^c\ar[d]\\
X\ar[r]^{\tau}\ar@{-->}[ur]^{\tilde{\tau}}  &  \BSO(2k).}
\]
As $\BSpin^c \to BSO(2k)$ is a $BU(1)$-principal bundle with the classifying map given by 
\[
W_3:  \BSO(2k) \to K(\ZZ, 3),
\]
 $E$ admits a $Spin^c$ structure if  $W_3\circ \tau: X\to K(\ZZ, 3)$ is null homotopic, and 
 a choice of null homotopy   determines   a $Spin^c$ structure on $E$.  Associated to a $Spin^c$ structure $\s$ on $E$, there is canonical  $K$-theoretical Thom class 
 \[
 U_E^\s= [\pi^*S^+, \pi^*S^-, cl]  \in  K^0_{cv}(E)
 \]
 in the   $K$-theory of $E$ with  vertical compact  supports. Here $S^+$ and $S^-$  are the positive and negative spinor bundle over $X$ defined by the $Spin^c$ structure on $E$, and $cl$ is
 the bundle map $\pi^*S^+ \to \pi^*S^-$ given by the Clifford action $E$ on $S^{\pm}$.  
 
\begin{remark} \begin{enumerate}
\item  The restriction of  $U_E^\s$  to each fiber is a generator of $K^0(R^{2k})$,  so     a  $Spin^c$ structure on $E$  is equivalent to  a K-orientation on $E$.  Note that Thom classes and  K-orientation 
are functorial under pull-backs of $Spin^c$ vector bundles. 
\item Let $\s\otimes L$ be  another $Spin^c$ structure on $E$ which differs from $\s$ by a
complex line bundle $p:  L \to X$, then 
\[
U_E^{\s'} = U_E^\s\cdot   p^*([L]).
\]
\item  Let $(E_1, \s_1)$ and $(E_2, \s_2)$ be two $Spin^c$ vector bundles over $X$, $p_1$ and $p_2$ be the projections from $E_1\oplus E_2$ to $E_1$ and $E_2$ respectively,  then 
\[  U_{E_1\oplus E_2}^{\s_1\oplus \s_2} = 
p^*_1( U_{E_1}^{\s_1} )  \cdot p_2^*(U_{E_2}^{\s_2})
 \in  K_{cv}^0(E_1\oplus E_2).
  \]

\item The Thom  isomorphism in  $K$-theory for a $Spin^c$  vector bundle $\pi: E\to X$ of rank $2k$ is given by
 \[
 \begin{array}{rcc}
\Phi^K_E:  \qquad   K^0(X) &\longrightarrow & K^0(E) \\
a &\mapsto  & \pi^*(a) U_E^\s. \end{array}
\]
Here for  locally compact spaces, we shall consider only  $K$-theory with compact supports. 
 When $X$ is compact,  $U_E^\s \in K^0(TX) $ and  the Thom
 isomorphism $\Phi_X^K$ is the inverse of the push-forward map 
 \[
 \pi_!: K^0(TX) \to K ^0(X)\]
associated to  the K-orientation of $\pi$ defined  the $Spin^c$ structure on $E$.  
 \end{enumerate}
 
\end{remark}

If an oriented vector bundle $E$ of even rank over $X$  does not admit a $Spin^c$ structure,   $W_3\circ \tau: X\to K(\ZZ, 3)$ is  not   null homotopic.  Thus,  $W_3\circ \tau$ defines a twisting  on $X$ for  $K$-theory, called the orientation twisting $o_E$. In this Section 
  we will define a canonical Thom class
\[
U_E \in K^0(E, \pi^*o_E) 
\]
such that $a\mapsto \pi^*(a)\cup U_E $ defines the Thom isomorphism $K^0(X, o_E) \cong K^0(E)$.
In fact, $a\mapsto \pi^*(a)\cup U_E $ defines the Thom isomorphism  (Cf. \cite{CW1}) 
\[
K^0(X, \a + o_E) \cong K^0(E, \a\circ \pi)
\]
for any twisting $\a: X\to K(\ZZ, 3)$.

Choose a good   open cover $\{V_i\}$ of $X$ such that $E_i = E|_{V_i}$ is trivialized by an isomorphism
\[
 E_i \cong V_i \times \RR^{2n}.
 \]
 This defines a canonical $Spin^c$ structure $\s_i$  on each $E_i$.  Denote by $U_{E_i}^{\s_i}$ the
 associated Thom class of $(E_i, \s_i)$. Then we have
 \[
 U_{E_j}^{\s_j} = U_{E_i}^{\s_i} \pi_{ij}^*([L_{ij}]) \in K_{cv}^0(E_{ij})
 \]
 where $L_{ij}$ is the difference line bundle over $V_{ij} =V_i\cap V_j$
  defined by  $\s_j = \s_i \otimes L_{ij}$   on $E_{ij} = E|_{V_{ij} }$.  Recall that these local line bundles
  $\{L_{ij}\}$ define a bundle gerbe \cite{Mur} associated to the twisting $ o_E = W_3\circ \tau: X\to K(\ZZ, 3)$ and
   a locally trivializing cover $\{V_i\}$ . By the definition of twisted  $K$-theory, $\{U_{E_i}^{\s_i}\}$ defines
   a twisted  $K$-theory class of $E$ with compact vertical supports and   twisting given by 
   \[
   \pi^* (o_E ) = o_E \circ \pi: E \to K(\ZZ, 3).
   \]
   We denote this canonical twisted  $K$-theory class by
   \[
   U_E \in K^0_{cv}(E,  \pi^* (o_E)  ).
   \]
   When $X$ is compact, then $ U_E \in K^0(E,  \pi^* (o_E )  )$.   One can easily show that
   the Thom class $U_E$ does not depend on the choice of the trivializing cover.

   Now we can list the properties of the Thom class in twisted  $K$-theory.
   
   \begin{proposition} \begin{enumerate}
\item If $E$ is equipped with a $Spin^c$ structure $\s$, then $\s$
defines a canonical isomorphism 
\[
\phi_\s:  K^0_{cv}(E, \pi^*(o_E ) ) \longrightarrow K^0_{cv}(E)
\]
such that $\phi_\s (U_E) = U_E^\s$.
\item Let $f: X\to Y$ be a continuous map and $E$ be  an oriented  vector bundle  of even rank over $Y$, then
\[
U_{f^*E} = f^*(U_E).
\]
\item Let $ E_1$ and $ E_2 $ be two oriented  vector bundles of even rank over $X$, $p_1$ and $p_2$ be the projections from $E_1\oplus E_2$ to $E_1$ and $E_2$ respectively, that is, we have the diagram
\[
\xymatrix{
E_1\oplus E_2 \ar[r]^{p_2} \ar[d]^{p_1} & E_2 \ar[d]^{\pi_2} \\
E_1 \ar[r]^{\pi_1} & X,}
\]
 then 
\[  U_{E_1\oplus E_2}  = 
p^*_1( U_{E_1}  )  \cdot p_2^*(U_{E_2} ).
\]
   
 \item Let $\pi: E \to X$ be an  oriented  vector bundle of even rank over a compact space $X$, the Thom isomorphism in twisted  $K$-theory (\cite{CW1}) 
\[
\Phi^K_E:  K^0(X, \a + o_E) \cong K^0(E, \pi^* (\a) )
\]
is given by  $a\mapsto \pi^*(a) \cdot  U_E $.  Moreover, the push-forward map  in twisted  $K$-theory (\cite{CW1}) 
\[
\pi_!:  K^0(E, \pi^* (\a) ) \longrightarrow K^0(X, \a + o_E) 
\]
satisfies $\pi_!( \pi^*(a) \cdot  U_E ) =a$.
\end{enumerate}
\end{proposition}

\begin{proof}
\begin{enumerate}
\item The $Spin^c$ structure $\s$ defines canonical isomorphism 
\[
\phi_\s:  K^0_{cv}(E, \pi^*(o_E) ) \longrightarrow K^0_{cv}(E)
\]
as follows.  Given a trivializing cover $\{V_i\}$ and the canonical $Spin^c$ structure $\s_i$ on $E_i = E|_{V_i}$,   we have 
\[
\s|_{E_i} = \s_i \otimes L_i
\]
for a complex line bundle $\pi_i: L_i \to V_i$. This implies $U_{E_i}^{\s} = U_{E_i}^{\s_i} \pi^*([L_i])$. Note that $L_{ij} = L_i  \otimes L_j^*$. 

Any twisted K-class  $a$ in  $K^0_{cv}(E, \pi^*(o_E ) )  $ is  given by 
a local K-class $a_i$ with
compact vertical support such that $a_j = a_i \pi_{ij}^*([L_{ij}])$,  then 
\[
a_i \pi_i^*([L_i]) = a_j \pi^*([L_j]) 
\]
in $K^0_{cv}(E|_{V_{ij}})$.  This  defines the homomorphism $\phi$, which is obviously an isomorphism
sending $U_E$ to $U_E^\s$.
\item Choose a good open cover $\{V_i\}$ of $Y$. 
By definition, the Thom class $U_E$ is defined by $\{U_{E_i}^{\s_i} \}$ with
\[
U_{E_j}^{\s_j} = U_{E_i}^{\s_i}\ \pi^*_{ij}([L_{ij}]).
\]
Then $\{f^{-1}(V_i)\}$ is an open cover of $X$, and $(f^*E)|_{f^{-1}(V_i)} = f^*E_i$
is trivialized with the canonical $Spin^c$ structure $f^*\s_i$,  thus
\[
U_{f^*E_i}^{f^*\s_i}  = f^* U_{E_i}^{\s_i}.
\] 
This  gives  $U_{f^*E} = f^*(U_E).$

\item The proof is similar to the proof of (2).
\item   From (\cite{CW1}), we know that the Thom isomorphism and the 
push-forward map  in twisted  $K$-theory  are both homomorphisms of $K^0(X, \a)$-modules. There exists
an oriented real vector bundle $F$  of even rank  such that 
\[
E\oplus F  = X\times \RR^{2m}
\]
for some $m\in \NN$.    Thus, we have 
\[
\xymatrix{
E\ar[r]^i\ar[dr]^\pi & X\times \RR^{2m} \ar[d]^p\\
& X}
\]
From the construction of the push-forward map in (\cite{CW1}), we see that
\[
\pi_! (U_E ) = p_! \circ i_! (U_E) = p_! (U_{E\oplus F}) =1.
\]
As the Thom isomorphism and the 
push-forward map  in twisted  $K$-theory  are both homomorphisms of $K^0(X, \a)$-modules, we get
$\pi_!( \pi^*(a) \cdot  U_E ) =a$.  

Note that the Thom isomorphism is inverse to the push-forward map  $\pi_!$, hence, the Thom isomorphism in twisted  $K$-theory   
\[
K^0(X, \a + o_E) \cong K^0(E, \pi^* (\a) )
\]
is given by  $a\mapsto \pi^*(a) \cdot  U_E $. 
\end{enumerate}

\end{proof}

\subsection{Twisted Riemann-Roch}
By an application of the Thom class and Thom isomorphism  in twisted  $K$-theory, we will now give a direct proof
of a special case of the Riemann-Roch theorem for twisted  $K$-theory. With  some notational changes, the   argument can be applied to establish 
the general Riemann-Roch theorem in twisted  $K$-theory.  Denote by $o_X$ and $o_Y$ 
the orientation twistings  associated to the tangent bundles $\pi_X: TX\to X$ and
$\pi_Y: TY \to Y$ respectively.

\begin{theorem} \label{RR}
Given  a   smooth map $f: X\to Y$ between oriented manifolds,  assume that $\dim Y- \dim X = 0\mod 2$. 
 Then  the Riemann-Roch formula is given by 
   \[
  Ch_{\check{w}_2(Y)} \bigl(f_{!}^{K}(a)\bigr) \hat{A} (Y)   =  f_*^H\bigl(Ch_{\check{w}_2(X) } (a)\hat{A} (X)\bigr).
\]
for any $a\in K^0 (X,    o_X ) $.  Here $\hat{A}(X)$ and $\hat{A}(Y)$ are  the A-hat classes of  $X$ and $Y$ respectively.  
\end{theorem}

\begin{proof} For simplicity, assume that both $X$ and $Y$ are of even dimension, say $2m$ and $2n$ respectively, equipped with a Riemannian metric. 
We will consider  Chern character defects in each of the following three squares 
\ba\label{square}\ \ \ 
\xymatrix{
K^0(X, o_X) \ar[r]^{\Phi_{TX}^K}_\cong \ar[d]_{Ch_{\check{w}_2(X)}} & K^0_c(TX) \ar[r]^{(df)^K_!}\ar[d]^{Ch}& K^0_c(TY) \ar[r]_\cong^{(\Phi_{TY}^K)^{-1}}  \ar[d]^{Ch}  &K^0(Y, o_Y) \ar[d]^{Ch_{\check{w}_2(Y)}} \\
H^{ev}(X) \ar[r]^{\Phi_{TX}^H}_\cong  & H^{ev}_c(TX)  \ar[r]^{(df)^H_*} & H^{ev}_c(TY) \ar[r]_\cong^{(\Phi_{TY}^H)^{-1}}  & H^{ev}(Y)
}
\na
where $\Phi_{TX}^K$ and $\Phi_{TY}^K$ are the Thom isomorphisms in twisted  $K$-theory for $TX$  and 
$TY$,  
$\Phi_{TX}^H$ and $\Phi_{TY}^H$ are the cohomology Thom isomorphisms for $TX$  and 
$TY$.  Then we have
\begin{enumerate}
\item The push-forward map in twisted  $K$-theory  as established in \cite{CW1}
\[
f_!^K:  K^0(X, o_X) \to K^0(Y, o_Y)
\]
agrees with 
\[
  (\Phi_{TY}^K)^{-1} \circ (df)^K_! \circ \Phi_{TX}^K.
\]
\item The push-forward map in cohomology theory  $f_*^H: H^{ev}(X) \to H^{ev}(Y)$  is given by 
\[
f_*^H = (\Phi_{TY}^H)^{-1} \circ (df)^H_* \circ \Phi_{TX}^H.
\]

\end{enumerate}

 Denote by $U_{TX}^H$ and $U_{TY}^H$ the cohomological Thom classes  for $TX$  and 
$TY$. Then  under the pull-back of the zero section, $0_X^*(U_{TX}^H) = e(TX)$ and $0_Y^*(U_{TY}^H) = e(TY)$  are the Euler classes for $TX$  and 
$TY$ respectively. 

  Let the Pontrjagin classes of $\pi_X: TX \to X$ be 
symmetric polynomials in $x_1^2, \cdots, x_m^2$, then 
\[
\hat{A}(X) = \prod_{k=1}^m \dfrac{x_k/2}{\sinh (x_k/2)}.  
\]
  The Chern character defect for  the left square in  (\ref{square}) is given by
\ba\label{defect:1}
Ch \bigl( \Phi^K_{TX} (a) \bigr) = \Phi_{TX}^H \bigl( Ch_{  \check{w}_2(X)}(a) \hat{A}^{-1}(X)\bigr) 
\na
for any $a \in K^0(X, o_X)$.  Here $\hat{A}(X)$ is the A-hat class of $TX$.  

To prove (\ref{defect:1}), note that 
\[\begin{array}{llll}
&&  (\Phi^H_{TX})^{-1} Ch \bigl( \Phi^K_{TX} (a) \bigr) & \\[2mm]
&=&  (\Phi^H_{TX})^{-1} \bigl( Ch (\pi_X^*(a)  \cdot U_{TX})  \bigr) 
 & \text{Apply Prop. (\ref{ring:homo})} \\[2mm]
&=&  (\Phi^H_{TX})^{-1} \bigl( Ch_{\pi_X^* \check{w}_2(X)}  (\pi_X^*(a) )  \cdot  Ch_{\pi_X^* \check{w}_2(X)}  (U_{TX})\bigr) &  \\[2mm] 
&=&  (\Phi^H_{TX})^{-1}  \bigl(   \pi^*_X (Ch_{ \check{w}_2(X)}   (a) )  \cdot  Ch_{\pi_X^* \check{w}_2(X)}  (U_{TX})\bigr)
&\text{Note that $(\Phi^H_{TX})^{-1}= (\pi_X )_*$. }  \\[2mm] 
&=& Ch_{  w_2(X)}    (a)  (\Phi^H_{TX})^{-1}  \bigl( Ch_{\pi_X^* \check{w}_2(X)}  (U_{TX})\bigr) &\text{By the projection formula.}
\end{array}
\]
So the  Chern character defect for  the  square  
\[\xymatrix{
K^0(X, o_X) \ar[r]^{\Phi_{TX}^K}_\cong \ar[d]_{Ch_{\check{w}_2(X)}} & K^0_c(TX)  \ar[d]^{Ch}\\
H^{ev}(X) \ar[r]^{\Phi_{TX}^H}_\cong  & H^{ev}_c(TX) }
\]
is given by 
\[
\cD(X) =  (\Phi^H_{TX})^{-1}    \bigl( Ch_{\pi_X^* \check{w}_2(X)}  (U_{TX})\bigr) \in H^{ev}(X).
\]
   From the cohomology Thom isomorphism, we have
\[
0_X^* \circ \Phi^H_{TX} (\cD(X) ) = \cD(X) e(TX),
\]
under the pull-back of the
 zero section $0_X$ of the tangent bundle $TX$. 

Therefore, we have
\[
\cD(X) = \dfrac{ 0_X^*  \bigl( Ch_{\pi_X^* \check{w}_2(X)}  (U_{TX})\bigr) }{e(TX)}.
\]
By the construction of the Thom class $U_{TX}$, under the pull-back of the
 zero section $0_X$ of $TX$, $0_X^* (U_{TX}) $ is a twisted K-class in
 $K^0(X, o_X)$ and 
\[\begin{array}{lll}
&& 0_X^* Ch_{\pi_X^* \check{w}_2(X)} (U_{TX}) \\[2mm]
& = & Ch_{  \check{w}_2(X)} ( 0_X^* (U_{TX})    )\\[2mm]
&=&  \prod_{k=1}^m (e^{x_k/2} - e^{-x_k/2}). \end{array}
\]
Thus, (\ref{defect:1}) follows from 
\[
\cD(X) = \prod_{k=1}^m \dfrac{ (e^{x_k/2} - e^{-x_k/2})}{ x_k} =   \hat{A}^{-1}(X).
\]
This implies that the following diagram commutes
\ba\label{diagram:1}
\xymatrix{
K^0(X, o_X) \ar[r]^{\Phi_{TX}^K}_\cong \ar[d]_{Ch_{\check{w}_2(X)} (-) \cdot \hat{A}(X) } & K^0_c(TX)  \ar[d]^{Ch (-)\cdot \pi^*_X  \hat{A}^2(X)}\\
H^{ev}(X) \ar[r]^{\Phi_{TX}^H}_\cong  & H^{ev}_c(TX) .}
\na

Similarly,   the  Chern character defect  in
 \[\xymatrix{
K^0(Y, o_Y) \ar[r]^{\Phi_{TY}^K}_\cong \ar[d]_{Ch_{\check{w}_2(Y)}} & K^0_c(TY)  \ar[d]^{Ch}\\
H^{ev}(Y) \ar[r]^{\Phi_{TY}^H}_\cong  & H^{ev}_c(TY) }
\]
is given by  
\[
Ch \bigl( \Phi^K_{TY} (a) \bigr) = \Phi_{TY}^H \bigl( Ch_{ \check{w}_2(Y)}(a) \hat{A}^{-1}(Y)\bigr) 
\]
for any $a\in K^0(Y, o_Y) $. 
This implies that the Chern character defect  for  the right square in  (\ref{square})  is given by
\ba\label{defect:3}
Ch_{\check{w}_2(Y)}\big( (\Phi^K_{TY})^{-1} (c) \bigr) \cdot \hat{A}^{-1}(Y)   =  (\Phi_{TY}^H)^{-1} \bigl(  Ch(c) \bigr) 
\na
for any $c\in K^0_c(TY) $.  Hence, we have the following commutative diagram
\ba\label{diagram:2}
\xymatrix{
K^0_c(TY)  \ar[d]_{Ch (-)\cdot \pi^*_Y  \hat{A}^2(Y)}   \ar[r]^{(\Phi_{TY}^K)^{-1}}_\cong & K^0(Y, o_Y)  \ar[d]^{Ch_{\check{w}_2(Y)}  (-)\cdot    \hat{A} (Y)}\\
H^{ev}_c(TY) \ar[r]^{(\Phi_{TY}^H)^{-1}}_\cong  & H^{ev} (Y) .}
\na

The Chern character for the middle  square in  (\ref{square})  follows from the Riemann-Roch theorem 
in ordinary  $K$-theory for the K-oriented map $df: TX  \to TY$ with the orientation given by
 canonical $Spin^c$  manifolds $TX$ and $TY$. Note that the Todd
classes of $Spin^c$ manifolds of $TX$ and $TY$ are given by  $\pi_X^*(\hat{A}^2(X))$
and $\pi_Y^*(\hat{A}^2(Y))$ respectively. This is due to two facts, that 
$T(TX) \cong \pi_X^*(TX\otimes \CC)$ and 
that $Td(TX\otimes \CC) = \hat{A}^2(X)$. So 
we have 
\ba\label{defect:2}
Ch\bigl((df)_!^K (a) \bigr) \cdot \pi_Y^* \bigl(\hat{A}^2(Y)\bigr) =(df)_*^H \bigl(Ch(a) \pi_X^* (\hat{A}^2(X)) \bigr)
\na
for any $a\in K^0_c(TX)$.   Hence, the following diagram commutes
\ba\label{diagram:3}
\xymatrix{
K^0_c(TX)  \ar[d]_{Ch (-)\cdot \pi^*_X  \hat{A}^2(X)} \ar[r]^{(df_!)^K}  & K^0_c(TY)  \ar[d]^{Ch (-)\cdot \pi^*_Y  \hat{A}^2(Y)}  \\
H^{ev}_c(TX)  \ar[r]^{(df)^H_*} & H^{ev}_c(TY).}
\na

Putting (\ref{diagram:1}), (\ref{diagram:2}) and (\ref{diagram:3}) together, we get the following 
commutative diagram
\[
\xymatrix{
K^0(X, o_X) \ar[r]^{f_!^K}  \ar[d]_{Ch_{\check{w}_2(X)} (-) \cdot \hat{A}(X) } &
K^0(Y, o_Y)  \ar[d]^{Ch_{\check{w}_2(Y)}  (-)\cdot    \hat{A} (Y)}\\
H^{ev}(X) \ar[r]^{f_*^H}   &  H^{ev} (Y) }
\]
which  leads to 
 \[
  Ch_{\check{w}_2(Y)} \bigl(f_{!}^{K}(a)\bigr) \hat{A} (Y)   =  f_*^H\bigl(Ch_{\check{w}_2(X) } (a)\hat{A} (X)\bigr)
\]
for any $a\in K^0 (X,    o_X ) $. This completes the proof of the Riemann-Roch theorem in twisted  $K$-theory.
\end{proof}

 With  some notational changes, the above  argument can be applied to establish 
the general Riemann-Roch theorem in twisted  $K$-theory.  
Let   $f : X \to Y$ a smooth map  between oriented manifolds with  $\dim Y- \dim X = 0\mod 2$. 
  Let $\check{\a} =(\cG_\a, \theta, \omega)$ 
be a differential twisting which lifts $\a: Y \to K(\ZZ, 3)$,  $f^*(\check{\a})$ is the pull-back 
differential twisting which lifts $\a\circ f : X \to K(\ZZ, 3)$.  Then we have
the  following Riemann-Roch formula
 \ba\label{RR:gen}
 Ch_{\check{\a}}  \bigl(f_{!}^{K}(a)\bigr) \hat{A} (Y)   =  
 f_*^H\bigl(Ch_{ f^* \check{\a} + \check{w}_2(Y)+f^*\check{w}_2(Y) } (a)\hat{A} (X)\bigr)
\na
for any $a\in K^0 (X,  \a\circ f +    o_X + f^* (o_Y) )$. 
  In particular, 
 we have  the following Riemann-Roch formula for a trivial twisting $\a: Y\to K(\ZZ, 3)$
 \ba\label{RR:AH}
  Ch \bigl(f_{!}^{K}(a)\bigr) \hat{A} (Y)   =  f_*^H\bigl(Ch_{\check{w}_2(Y)+f^*\check{w}_2(Y) } (a)\hat{A} (X)\bigr)
\na
for any $a\in K^0 (X,    o_X + f^* (o_Y) )$. 

When $f$ is K-oriented, and equipped  with a  $Spin^c$ structure whose determinant bundle is $L$, there is a canonical isomorphism 
\[
\Psi:  K^0(X) \cong K^0(X,    o_X + f^* (o_Y) )
\]
such that $Ch_{\check{w}_2(Y)+f^*\check{w}_2(Y) } (\Psi(a)) = e^{c_1(L)/2} Ch (a)$ for any $a\in K^0(X)$.  
Then the Riemann-Roch  formula  (\ref{RR:AH}) agrees with the Riemann-Roch  formula for K-oriented
maps as established in \cite{AH}.

\section{The twisted index formula}

In this Section, we establish the index pairing for a closed smooth manifold with a twisting $\a$
\[
K^{ev/odd}(X, \a)  \times \hKa_{ev/odd} (X, \a) \longrightarrow \ZZ
\]
in terms of the local index formula for twisted geometric cycles. 

\begin{theorem}\label{index}
Let $X$ be a smooth closed  manifold 
with a twisting $\a: X\to K(\ZZ, 3)$.  The index pairing
\[
K^0(X, \a) \times K_0(X, \a) \longrightarrow \ZZ
\]
is given by
\[  
     \<  \xi, (M, \i, \nu, \eta, [E]) \>  \\[2mm]
    = \displaystyle{  \int_M}  Ch_{\check{w}_2(M)} \bigl(\eta_* (\i^*\xi\otimes E) \bigr) \hat{A}(M) 
 \]
where $\xi\in K^0(X, \a)$, and the geometric cycle 
$
(M, \i, \nu, \eta, [E])
$
 defines a  twisted $K$-homology class on $(X, \a)$.  Here 
\[
 \eta_* : K^*(M, \i^*\a) \cong K^*(M, o_M) 
 \]
 is an  isomorphism, and $Ch_{\check{w}_2(M)}$ is the 
  Chern character on  $K^0(M,o_M)$.
  \end{theorem}
    \begin{proof} Recall that the index pairing $K^0(X, \a) \times K_0(X, \a) \longrightarrow \ZZ$ can be defined  by  the internal  Kasparov product  (Cf. \cite{Kas2} and \cite{ConSka}) 
    \[
    KK(\CC, C(X, \cP_\a(\cK)) ) \times KK(C(X, \cP_\a(\cK)), \CC) \longrightarrow KK(\CC, \CC) \cong \ZZ,
    \]
      and is functorial in the sense that if $f: Y\to X$ is a continuous map and
    $Y$ is equipped with a twisting $\a: X\to \ZZ$ then
    \[
    \<  f^*b, a \> = \<b,  f_*(a)\>
    \]
    for any $a\in K_0(Y, f^*\a)$ and $b\in K^0(X, \a)$. 
    
    Note that under the assembly map, the geometric cycle 
$
(M, \i, \nu, \eta, [E])
$
 is mapped to $\i_* \circ \eta_* ([M]\cap [E])$, for $\xi \in K^0(X, \a)$. Hence, we have
 \[
 \begin{array}{lll}
&&   \<  \xi, (M, \i, \nu, \eta, [E]) \>\\[2mm]
&=& \< \xi, \i_* \circ \eta_* ([M]\cap [E])\> \\[2mm]
&=& \<\i^*\xi, \eta_*([M]\cap [E])\> \\[2mm]
&=&   \<\eta_*(\i^*\xi \otimes E), [M]\>.\end{array}
\]
Here $\eta_*(\i^*\xi \otimes E) \in K^0(M, o_M)$ and $[M]$ is the fundamental class 
in $\hKa_{ev}(M, o_M)$ which is Poincar\'e dual to the unit element $\underline{\CC}$ in $K^0(M)$. The index 
pairing between $K^0(M, o_M) \times \hKa_{ev}(M, o_M)$ can be written as
\[
K^0(M, o_M) \times \hKa_{ev}(M, o_M) \to K^0(M, o_M) \times  K^0(M) \to K^0(M, o_M) \to \ZZ
\]
where the first map is given by the Poincar\'e duality $\hKa_{ev}(M, o_M) \cong K^0(M)$, the
middle map is the action of $K^0(M)$ on $K^0(M, o_M)$, and the last map is
the push-forward map of $\eps: M\to pt$. Therefore, we have
\[
 \begin{array}{lll}
&&  \<\eta_*(\i^*\xi \otimes E), [M]\> \\[2mm]
&=& \eps^K_!\bigl(  \eta_*(\i^*\xi \otimes E) \otimes \underline{\CC} \bigr)\\[2mm]
&=&   \eps^K_!\bigl(  \eta_*(\i^*\xi \otimes E)  \bigr).\end{array}
\]
By twisted Riemann-Roch (Theorem \ref{RR}), 
\[
 \begin{array}{lll}
&& \eps_!\bigl(  \eta_*(\i^*\xi \otimes E)  \bigr) \\[2mm]
&=&  \eps^H_* \bigl( Ch_{\check{w}_2(M)} \bigl(\eta_* (\i^*\xi\otimes E) \bigr)  \hat{A}(M) \bigr)  \\[2mm]
&=& \displaystyle{ \int_M}  Ch_{\check{w}_2(M)} \bigl(\eta_* (\i^*\xi\otimes E) \bigr) \hat{A}(M). 
\end{array}
 \]
This completes the proof of the twisted index formula.
 \end{proof}

Note that $\eps: M\to pt$ can be written as $\i\circ\eps_X: M\to X\to pt$.   Applying the Riemann-Roch Theorem \ref{RR}, we can write the above index pairing as
  \[\begin{array}{lll}
    && < (M, \i, \nu, \eta, [E]), \xi > \\[2mm]
     &=& \displaystyle{  \int_M}  Ch_{\check{w}_2(M)} \bigl(\eta_* (\i^*\xi\otimes E) \bigr) \hat{A}(M)\\[4mm]
    &=&  \displaystyle{ \int_X} Ch_{\check{w}_2(X)} \bigl(\ \i_!(E) \otimes \xi \bigr) \hat{A}(X)
    \end{array}
 \]
where $\i_!: K^0(M) \to K^0(X, -\a+o_X)$ is the push-forward map in twisted  $K$-theory, \[
K^0(X, \a ) \times K^0(X, -\a+o_X) \longrightarrow  K^0(X, o_X)  
\]
is the multiplication map (\ref{multi:add}), 
and
 \[
 Ch_{\check{w}_2(X)}: K^0(X, o_X) \longrightarrow H^{ev}(X) 
 \]
 is the twisted Chern character (which agrees with the relative Chern character under the identification $K^0(X, o_X) \cong K^0(X, W_3(X))$, the  $K$-theory of Clifford modules on $X$).

  \section{Mathematical definition of D-branes and D-brane charges}
 
Here we give a mathematical interpretation of D-branes in Type II string theory using the twisted geometric cycles and use the index theorem in the previous Section to compute  charges of D-branes.  
  In Type II superstring theory on a manifold  $X$,  a string worldsheet is an  oriented Riemann surface
  $\Sigma$, mapped into $X$ with $\pa \Sigma$ mapped to an oriented submanifold $M$ (called a D-brane world-volume, a source  of the Ramond-Ramond flux). The theory also has  a  Neveu-Schwarz $B$-field classified by a characteristic class $[\a]\in H^3(X, \ZZ)$.  
  
  In physics,  the D-brane  world volume $M$ carries a gauge field on a complex vector bundle
 (called the Chan-Paton bundle), so a D-brane is given by a
 submanifold $M$ of $X$ with a complex bundle $E$  and a connection $\nabla^E$. This data actually defines a  differential K-class 
 \[
 [(E, \nabla^E)]
 \]
 in differential   $K$-theory  $\check{K}(M)$.

  When the $B$-field is  topologically trivial, that is
  $[\a]=0$,  D-brane charge takes values in ordinary $K$-theory 
$K^0(X )$  or $K^1(X )$  for Type IIB or Type IIA string theory  (as explained in \cite{MM}\cite{Wit1}).   
 For a D-brane  $M$  to define a class in the $K$-theory of $X$, its normal bundle $\nu_M$
 must be endowed with a $Spin^c$ structure.   Equivalently, the embedding 
 \[
 \i: M \longrightarrow  X
 \]
 is K-oriented so that the push-forward map in  $K$-theory (\cite{AH})
 \[
 \i_!^K:  K^0(M) \longrightarrow  K^{ev/odd}(X)
 \]
 is well-defined, (it takes values in even or odd K-groups depending on the dimension of $M$).    So the D-brane charge  of
 $(\i: M \to X, E)$ is  
 \[
 \i_!^K ([E]) \in K^{ev/odd}(X).
 \]
 It was proposed  in \cite{MM} that the cohomological Ramond-Ramond charge of the
 D-brane   is given by
 \[
\cQ_{RR} (\i: M \to X, E) =
 ch(f_!^K (E))\sqrt{\hat{A}(X)}
 \]
 when $X$ is a Spin manifold. 
 A natural  $K$-theoretic interpretation follows from the fact that the modified
 Chern  character isomorphism 
 \[
 K^{ev/odd}_\QQ (X)  \longrightarrow H^{ev/odd}(X, \QQ)
 \]
given by mapping $a\mapsto ch(a)\sqrt{\hat{A}(X)}$ is an isometry with the natural bilinear parings on
 $K^*_\QQ(X) = K^*(X)\otimes \QQ$ and $H^{ev/odd}(X, \QQ)$. Here the pairing on
$K(X)$  is given by the index of the Dirac operator
 \[
 (a, b)_K = \ind (\Dirac_{a\otimes b}) = \int_X ch(a)ch(b)  \hat{A}(X) = \bigl(ch(a) \sqrt{\hat{A}(X)},
(ch(b) \sqrt{\hat{A}(X)}\bigr)_H.
 \]

  When the  $B$-field is not  topologically trivial, that is
  $[\a]\neq 0$,  then $[\a]$   defines a complex line bundle over the loop space $LX$, or a stable isomorphism class of bundle gerbes over $X$. Then in order to have a well-defined worldsheet path integral, Freed and Witten in \cite{FreWit} showed that 
  \ba\label{FW:condition}
  \i^* [\a]  + W_3(\nu_M) =0.
  \na
  When  $ \i^* [\a]  \neq 0$, that means $\i$ is not K-oriented, then
  the push-forward map in  $K$-theory (\cite{AH})
 \[
 \i_!^K:  K^0(M) \longrightarrow  K^*(X)
 \] 
 is {\bf not}  well-defined.    Witten explained in \cite{Wit1} that D-brane
charges should take values in a twisted form of  $K$-theory, as supported further by evidence in \cite{BouMat} and
 \cite{Kap}.
 
 In \cite{Wang}, the mathematical meaning of  (\ref{FW:condition}) was discovered using the notion of  $\a$-twisted $Spin^c$ manifolds for a continuous map 
 \[
 \a: X\longrightarrow K(\ZZ, 3)
 \]
 representing $[\a] \in H^3(X, \ZZ)$. When $X$ is $Spin^c$, the datum to describe a D-brane is exactly a geometric cycle for the twisted $K$-homology $\hKg_{ev/odd}(X, \a)$. By  Poincar\' e duality, we have
 \[
 \hKg_{ev/odd}(X, \a) \cong K^0(X, \a+o_X)
 \]
 with the orientation twisting $o_X: X\to K(\ZZ, 3)$  trivialized by a choice
 of a $Spin^c$ structure. Hence,  
 \[
 K^0(X, \a+o_X) \cong K^0(X, \a).
 \]
 
 For a general manifold $X$, a submanifold $\i: M\to X$ with
 \[
\i^*([\a]) + W_3(\nu_M) =0,
\]
then there is a homotopy commutative diagram
\[
\xymatrix{M \ar[d]_{\i} \ar[r]^{\nu_M} & 
\BSO
 \ar@2{-->}[dl]_{\eta} \ar[d]^{W_3} \\
X \ar[r]_\a  & K(\ZZ, 3), } 
\]
here $\nu_M$ also denotes a  classifying map of the normal bundle, or a classifying map of
the bundle $TM\oplus \i^*TX$. This motivates the following definition (see also \cite{CW2}).

\begin{definition} \label{D:brane} Given  a smooth manifold $X$ with a twisting $\a: X\to K(\ZZ, 3)$, a $B$-field of $(X, \a)$ is a differential twisting  lifting $\a$
\[
 \check{\a} = (\cG_\a, \theta, \omega),
 \]
 which is a (lifting, or local)  bundle gerbe $\cG_\a$ with a connection $\theta$ and a curving $\omega$. 
 The field strength of the $B$-field $(\cG_\a, \theta, \omega)$ is given by the curvature  $H$ of $\check{\a}$. 
 
 A Type  II  (generalized) D-brane in $(X, \a)$ is a complex vector bundle $E$ with a connection 
 $\nabla^E$  over a twisted $Spin^c$  manifold $M$. The twisted $Spin^c$ structure on $M$ is given by the following homotopy commutative diagram together  with a choice of a homotopy $\eta$
 \ba\label{D-brane}
 \xymatrix{M \ar[d]_{\i} \ar[r]^{\nu_\i} & 
\BSO
 \ar@2{-->}[dl]_{\eta} \ar[d]^{W_3} \\
X \ar[r]_\a  & K(\ZZ, 3) } 
\na
where $\nu_\i$ is the classifying map of  $TM\oplus \i^*TX$.
\end{definition}
 
 \begin{remark}
  The twisted $Spin^c$ manifold $M$ in Definition \ref{D:brane} is the D-brane world volume
in Type II string theory.  The twisted $Spin^c$ structure given in (\ref{D-brane}) implies
that D-brane world volume $M \subset X$ 
in Type II string theory satisfies the Freed-Witten anomaly cancellation condition
\[
\i^*[\a] + W_3(\nu_M) = 0.
\]
  In particular, if the $B$-field of $(X, \a)$  is topologically trivial, then the normal bundle of $M\subset X$ is equipped with a $Spin^c$ structure given by (\ref{D-brane}). 
  \end{remark}
 
Given a Type II D-brane $(M, \i, \nu_\i, \eta,  E, \nabla^E)$, the homotopy $\eta$ induces an isomorphism
\[
\eta_*: K^0(M) \to K^0(M, \i^*\a + o_\i).\]
Here $o_\i$ denotes the orientation twisting of the bundle $TM\oplus \i^*TX$.
Note that  \[
\i_!^K: K^0(M, \i^*\a + o_\i) \longrightarrow  K^{en/odd} (X, \a) 
\]
  is the  pushforward map (\ref{push:forward}) in twisted  $K$-theory.  Hence we have  a canonical  element in $ K^{en/odd} (X, \a) $ defined by
\[
\i_!^K(\eta_* ([E])),  
\]
 called the D-brane charge of $(M, \i, \nu_\i, \eta,  E)$.   We remark that a Type II D-brane 
 \[
 (M, \i, \nu_\i, \eta,  E, \nabla^E)
 \]
  defines an element in differential twisted  $K$-theory 
 $\check{K}^{en/odd} (X, \check{\a})$.
 
 From (\ref{D-brane}), we know that $M$ is an $(\a+o_X)$-twisted $Spin^c$ manifold as we have the following homotopy commutative diagram
 \[\xymatrix{M \ar[d]_{\i} \ar[r]^{\nu} & 
\BSO
 \ar@2{-->}[dl]  \ar[d]^{W_3} \\
X \ar[r]_{\a+o_X}  & K(\ZZ, 3)  } 
\]
where $\nu$ is the classifying map of the stable normal bundle of $M$. 
 Together with the following proposition, we conclude that  the Type II D-brane charges,  in the present of a 
 $B$-field  
 \[
 \check{\a} = (\cG_\a, \theta, \omega),
 \]
 are classified by twisted  $K$-theory  $K^0(X, \a)$. 
 
\begin{proposition}  Given a twisting $\a: X\to K(\ZZ, 3)$ on a smooth manifold $X$,  every twisted
K-class in $K^{ev/odd}(X, \a)$ is represented by a geometric cycle
supported on an $(\a + o_X)$-twisted closed $Spin^c$-manifold
$M$ and an ordinary  K-class $[E] \in K^0(M)$.
\end{proposition}

For completeness, we also give a definition of Type I D-branes (Cf.  \cite{MMS}, \cite{RSV} and  Section 8 in \cite{Wang}).

\begin{definition} Given  a smooth manifold $X$ with a KO-twisting $\a: X\to K(\ZZ_2, 2)$, 
 a Type I  (generalized) D-brane in $(X, \a)$ is a real vector bundle $E$ with a connection 
 $\nabla^E$  over a twisted $Spin$  manifold $M$. The twisted $Spin$ structure on $M$ is given by the following homotopy commutative diagram together  with a choice of a homotopy $\eta$
 \ba\label{D-brane:I}
 \xymatrix{M \ar[d]_{\i} \ar[r]^{\nu_\i} & 
\BSO
 \ar@2{-->}[dl]_{\eta} \ar[d]^{w_2} \\
X \ar[r]_\a  & K(\ZZ_2, 2) } 
\na
where $w_2$ is the classifying map of the principal 
$K(\ZZ_2, 1)$-bundle $\BSpin  \to \BSO$  associated to
the second   Stiefel-Whitney class,    $\eta$ is a homotopy
between $w_2 \circ \nu_\i $ and $\a \circ \i$.   Here  $\nu_\i$ is the classifying map of  $TM\oplus \i^*TX$.
\end{definition}

\begin{remark}  A Type I D-brane in $(X, \a)$ has its support on a manifold $M$ if and only if
there is a differentiable map $\i: M\to X$ such that
\[
\i^*([\a] ) + w_2(\nu_\i) =0.
\]
Here $\nu_\i$ denotes the bundle $TM\oplus  \i^*TX$.  
\end{remark} 

Given a Type I D-brane in $(X, \a)$,   the  push-forward map in twisted KO-theory
\[
\xymatrix{
KO^*(M) \ar[r]^{\eta_* \qquad }_{\cong\qquad} & KO^*(M, \a\circ \i + o_\i)\ar[r]^{\qquad \i_!^{KO}  }&  KO^*(X, \a)}
\]
 defines  a canonical element  in $KO^*(X, \a)$. Every class
in $KO^*(X, \a)$ can be realized by a Type I  (generalized) D-brane in $(X, \a)$. Hence, we conclude that  the Type I D-brane charges  are classified by twisted KO-theory  $KO^{ev/odd}(X, \a)$.

  \end{document}